\documentclass[11pt]{article}

\usepackage[english]{babel}

\usepackage[utf8]{inputenc}
\usepackage{amsfonts,amssymb,amsmath,amsthm,latexsym,epsfig,amscd,bbm,stmaryrd,mathrsfs,hyperref,fancyhdr}
\usepackage{graphicx,color,tikz,courier,bbm,enumerate,psfrag,wasysym,type1cm,dsfont,geometry,marginnote,afterpage}

 \usepackage{pst-all}

\numberwithin{equation}{section}

\theoremstyle{plain} 

\newtheorem{theorem}{Theorem}[section]

\newtheorem{definition}[theorem]{Definition}

\newtheorem{remark}[theorem]{Remark}
\newtheorem{example}[theorem]{Example}
\newtheorem{conjecture}[theorem]{Conjecture}

\newcommand{\N}{\mathbb N}
\newcommand{\R}{\mathbb R}

\newcommand{\dd}{\mathrm{d}}
\newcommand{\ee}{\mathrm{e}}

\newcommand{\prob}{\overset{\mathbb{P}}{\longrightarrow}}
\newcommand{\weak}{\Longrightarrow}

\newcommand{\E}{\mathbb E}
\renewcommand{\P}{\mathbb P}

\newcommand{\HER}{\mathrm{HER}}
\newcommand{\IER}{\mathrm{IER}}
\newcommand{\CM}{\mathrm{CM}}
\newcommand{\PAM}{\mathrm{PAM}}

\newcommand{\cF}{\mathcal{F}}

\definecolor{mydarkblue}{RGB}{0,0,139}

\newcommand{\no}{N^{(\textsf{old})}}
\newcommand{\ny}{N^{(\textsf{young})}}

\usepackage[normalem]{ulem} 
\usepackage{cancel}

\begin{document}

\title{The friendship paradox for sparse random graphs}

\author{
\renewcommand{\thefootnote}{\arabic{footnote}}
Rajat Subhra Hazra, Frank den Hollander, Azadeh Parvaneh
\footnotemark[1]
}

\footnotetext[1]{
Mathematical Institute, Leiden University, Einsteinweg 55, 2333 CC Leiden, The Netherlands.\\ 
{\tt \{r.s.hazra,denholla,s.a.parvaneh.ziabari\}@math.leidenuniv.nl}
}

\maketitle

\abstract{Let $G_n$ be an undirected finite graph on $n\in\N$ vertices labelled by $[n] = \{1,\ldots,n\}$. For $i \in [n]$, let $\Delta_{i,n}$ be the \emph{friendship bias} of vertex $i$, defined as the difference between the average degree of the neighbours of vertex $i$ and the degree of vertex $i$ itself when $i$ is not isolated, and zero when $i$ is isolated. Let $\mu_n$ denote the \emph{friendship-bias empirical distribution}, i.e., the measure that puts mass $\frac{1}{n}$ at each $\Delta_{i,n}$, $i \in [n]$. The friendship paradox says that $\int_\R x\mu_n(\dd x) \geq 0$, with equality if and only if in each connected component of $G_n$ all the degrees are the same.  

We show that if $(G_n)_{n\in\N}$ is a sequence of sparse random graphs that converges to a rooted random tree in the sense of convergence locally in probability, then $\mu_n$ converges weakly to a limiting measure $\mu$ that is expressible in terms of the law of the rooted random tree. We study $\mu$ for four classes of sparse random graphs: the homogeneous Erd\H{o}s-R\'enyi random graph, the inhomogeneous Erd\H{o}s-R\'enyi random graph, the configuration model and the preferential attachment model. In particular, we compute the first two moments of $\mu$, identify the right tail of $\mu$, and argue that $\mu([0,\infty))\geq\tfrac12$, a property we refer to as \emph{friendship paradox significance}.  
}



\medskip\noindent
\emph{Acknowledgement:}
The research in this paper was supported through NWO Gravitation Grant NETWORKS 024.002.003. AP has received funding from the European Union's Horizon 2020 research and innovation programme under the Marie Sk\l odowska-Curie grant agreement Grant Agreement No 101034253. The authors are grateful to Rob van den Berg, Remco van der Hofstad and Nelly Litvak for helpful discussions.

\vspace{0.1cm}
\hfill\includegraphics[scale=0.1]{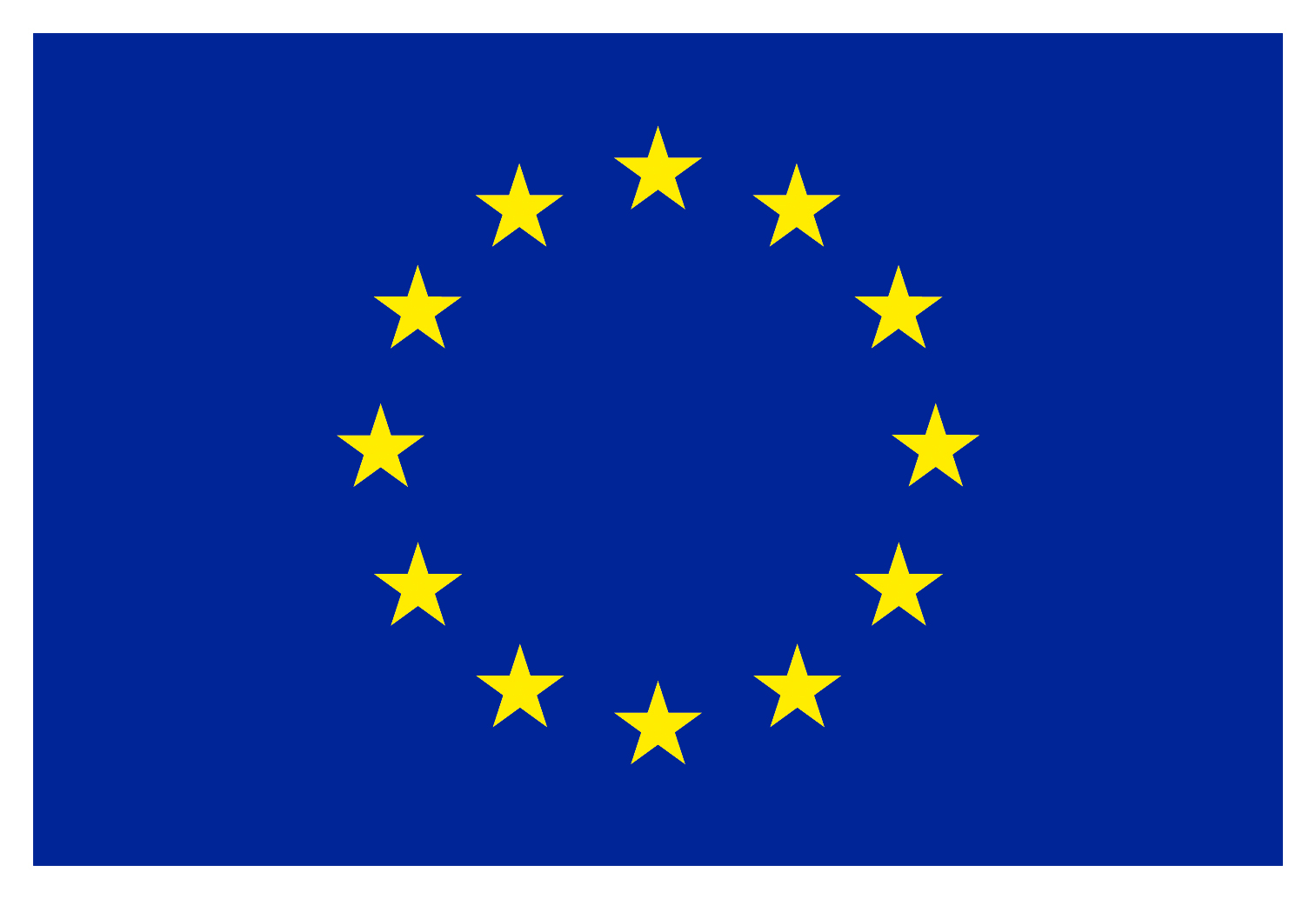}


\newpage
\small
\tableofcontents
\normalsize
\newpage


\section{Introduction and outline}
\label{sec:intro}


\subsection{Background and motivation}

In 1991, the American sociologist Scott Feld discovered the paradoxical phenomenon that `your friends are more popular than you' \cite{SF}.  This statement means the following. Consider a group of individuals who form a connected friendship network. For each individual in the group, compute the difference between the average number of friends of friends and the number of friends (all friendships in the group are mutual). Average these numbers over all the individuals in the group. It turns out that the latter average is always non-negative, and is strictly positive as soon as not all individuals have exactly the same number of friends. This bias, which at first glance seems counterintuitive, goes under the name of \emph{friendship paradox}, even though it is a hard fact. An equivalent, and possibly more soothing, version of the paradox reads `your enemies have more enemies than you' \cite{PMD}.


\paragraph{Implications.}

Apart from being interesting in itself, the friendship paradox has useful implications. For instance, it can be used to slow down the spread of an infectious disease. Suppose that there is a group of individuals whose friendship network we do not know explicitly. Suppose that an infectious disease breaks out and we only have one vaccine at our disposal, which we want to use as effectively as possible. In other words, we want to vaccinate the individual who has the most friends. One way we could do this is by choosing an individual at random and giving the vaccine to him or her. Another way could be to select an individual at random, and let him or her choose a friend at random to give the vaccine to. Since the more popular individuals are more likely to be chosen, the second approach is more effective in combating the disease.

The friendship bias can be viewed as a \emph{centrality measure}, akin to PageRank \cite{PBMW} centrality and degree centrality. In \cite{GRL2020}, numerous intriguing features of PageRank were explored in the context of sparse graphs with the help of local weak convergence. Similar interesting features emerge in the analysis of the friendship bias. In \cite{CKN} and \cite{PYNSB}, the behaviour of the friendship paradox for random graphs like the Erd\H{o}s-R\'enyi random graph was studied. Through an empirical analysis of the limit as $n\to\infty$ of Erd\H{o}s-R\'enyi random graphs with $n$ vertices and connection probability $\lambda/(n-1)$ conditioned not to have isolated points, \cite{CKN} concluded that ``for large $\lambda$ no meaningful friendship paradox applies" (a statement that we will refute in the present paper). Moreover, \cite{CKN} examined a special case of the configuration model and provided an empirical analysis by using kernel density estimation. Furthermore, a random network model was proposed with degree correlation and, with the help of the notion of Shannon entropy, an attempt was made to investigate the impact of the assortativity coefficient on the friendship paradox.
 
There have also been studies on what is called the \emph{generalised friendship paradox}, in which attributes other than popularity are considered that produce a similar paradox \cite{EJ}. For instance, an analysis of two co-authorship networks of Physical Review journals and Google Scholar profiles reveals that on average the co-authors of a person have more collaborations, publications and citations than that person \cite{EJ}. On Twitter, for most users on average their friends share and tweet more viral content, and over 98 percent of the users have fewer followers than the people whom they followed \cite{HKL, J}. In \cite{CKN}, it is shown in an informal way that the generalised friendship paradox holds when the attribute correlates positively with popularity. Other works have examined the implications of the friendship paradox for individual biases in perception and thought contagion, noting that our social norms are influenced by our perceptions of others, which are strongly shaped by the people around us \cite{J}. For instance, individuals whose acquaintances smoke are more likely to smoke themselves \cite{CF}. The impact of the friendship paradox on behaviour explains why we should expect more-connected people to behave systematically differently from less-connected people, which in turn biases overall behaviour in society \cite{J}. In \cite{NK}, a new strategy based on the friendship paradox is imposed to improve poll predictions of elections, instead of randomly sampling individuals and asking such questions as `Who are you voting for? Who do you think will win?' In \cite{CR}, the friendship paradox is looked at from a probabilistic point of view, and it is shown that a randomly chosen friend of a randomly chosen individual stochastically has more friends than that individual.


\paragraph{Mathematical modelling.}

The friendship paradox raises many questions. `How large is the bias? How does it depend on the architecture of the friendship network? Are there universal features beyond the fact that the bias is always non-negative?' Mathematically, the friendship structure is modelled by a graph, where the vertices represent individuals and the edges represent friendships. The graph is typically random and its size is typically large. In the present paper we take a \emph{quantitative} look at the friendship paradox for \emph{sparse random graphs}. We focus on the \emph{friendship-bias empirical distribution}, i.e., the distribution of the biases of all the individuals in the network. We show that if a sequence of sparse random graphs converges to a rooted random tree in the sense of convergence locally in probability, then the friendship-bias empirical distribution converges weakly to a limiting measure that is expressible in terms of the law of the rooted random tree. We study this limiting measure for four classes of sparse random graphs. In particular, we compute its first two moments, identify its right tail, and argue that it puts at least one half of its mass on non-negative biases, a property we refer to as \emph{friendship paradox significance}.


\subsection{Friendship paradox}
\label{FPintro}


\paragraph{Random graphs.}

Throughout the sequel, $G = (V(G),E(G))$ is an undirected simple graph or multi-graph. The vertices $V(G)$ represent individuals, the edges $E(G)$ represent mutual friendships. Denote by
\begin{align*}
(d_{i}^{(G)})_{i \in V(G)}, \qquad A^{(G)} = (A_{i,j}^{(G)})_{i,j\in V(G)},
\end{align*}
the \emph{degree sequence}, respectively, the \emph{adjacency matrix} of $G$, where $A_{i,j}^{(G)}$ is the number of edges between vertices $i$ and $j$ when $j\neq i$, and twice the number of self-loops at vertex $i$ when $j=i$ (i.e., each self-loop adds 2 to the degree).

Let $G_{n}$ be a finite graph on $n\in\N$ vertices labelled by $[n] = \{1,\ldots,n\}$. For $i \in [n]$, define the \emph{friendship bias} as
\begin{align*}
\Delta_{i,n} = \left[\dfrac{\sum_{j \in [n]} A_{i,j}^{(G_{n})} d_{j}^{(G_{n})}}{d_{i}^{(G_{n})}} - d_{i}^{(G_{n})}\right]
\,\mathbbm{1}_{\{d_{i}^{(G_{n})} \neq 0\}}.
\end{align*}
Let $\mu_{n}\colon\,\mathcal{B}(\R) \to [0,1]$ be the \emph{quenched friendship-bias empirical distribution} 
\begin{align*}
\mu_{n} = \dfrac{1}{n}\sum_{i \in [n]} \delta_{\Delta_{i,n}},
\end{align*}
where $\delta_{x}$ denotes the Dirac measure concentrated at $x$, and $\mathcal{B}(\R)$ the Borel $\sigma$-algebra on $\R$. Let $\tilde{\mu}_{n}\colon\,\mathcal{B}(\R) \to [0,1]$ be the \emph{annealed friendship-bias empirical distribution} defined by 
\begin{equation*}
\tilde{\mu}_{n}(\cdot) = \E_{n}[\mu_{n}(\cdot)],
\end{equation*}
where $\E_{n}$ denotes expectation with respect to the random graph $G_{n}$. 

The \emph{average friendship bias} is defined as
\begin{equation*}
\Delta_{[n]} = \frac{1}{n} \sum_{i \in [n]} \Delta_{i,n} = \int_{\R} x\mu_{n}(\dd x). 
\end{equation*}
The choice to count each self-loop as contributing $2$ to the degree of the vertex guarantees that $\sum_{j \in [n]} A_{i,j}^{(G_n)} = d_i$. Hence we can write
\begin{align*}
\Delta_{[n]} &= \dfrac{1}{n} \sum_{ {i \in [n]} \atop {d_{i}^{(G_{n})} \neq 0} } \sum_{ { j \in [n]} \atop {d_{j}^{(G_{n}) } \neq 0}} A_{i,j}^{(G_{n})}
\left(\dfrac{d_{j}^{(G_{n})}}{d_{i}^{(G_{n})}}-1\right)\\
&= \dfrac{1}{2n} \sum_{ {i \in [n]} \atop {d_{i}^{(G_{n})} \neq 0} } \sum_{ { j \in [n]} \atop {d_{j}^{(G_{n}) } \neq 0}}A_{i,j}^{(G_{n})}
\left(\sqrt{\dfrac{d_{j}^{(G_{n})}}{d_{i}^{(G_{n})}}}-\sqrt{\dfrac{d_{i}^{(G_{n})}}{d_{j}^{(G_{n})}}}\,\right)^{2}
\geq 0,
\end{align*}
with equality if and only if all connected components of $G_{n}$ are regular. This property is known as the \emph{friendship paradox}, because if the edges of the graph represent mutual friendships, then $\Delta_{[n]}>0$ means that in a community with $n$ individuals on average the friends of an individual have more friends than the individual itself. 


\paragraph{Rooted random graphs.}

In the sequel we will consider a sequence $(G_n)_{n\in\N}$ of finite random graphs that locally converges to an almost surely locally finite, connected \emph{rooted random graph} $G$ with root vertex $o$, denoted by the pair $(G,o)$. The friendship-bias of the root is
\begin{align}
\label{Dphidefalt}
\Delta_{o}^{(G)} =\bigg[\dfrac{\sum_{j\in V(G)} A_{o,j}^{(G)} d_{j}^{(G)}}{d_{o}^{(G)}}-d_{o}^{(G)}\bigg]\,
\mathbbm{1}_{\big\{d_{o}^{(G)}\neq 0\big\}}.
\end{align}
We write $\mu\colon\,\mathcal{B}(\R) \to [0,1]$ to denote the law of $\Delta_{o}^{(G)}$.


\paragraph{Outline.} 

In Section~\ref{sec:locconv} we use the theory of local convergence of sparse random graphs to show that the friendship-bias empirical distribution $\mu_n$ converges to a limit $\mu$ as $n\to\infty$. The convergence is stated in two theorems, one for the quenched empirical distribution and one for the annealed empirical distribution. In Section~\ref{sec:proofs} we give the proof of these two theorems. In Section~\ref{sec:ex} we study $\mu$ for four classes of sparse random graphs: the homogeneous Erd\H{o}s-R\'enyi random graph, the inhomogeneous Erd\H{o}s-R\'enyi random graph, the configuration model and the preferential attachment model. For each of the four classes we state three theorems, in which we compute the first two moments of $\mu$, identify the right tail of $\mu$, and argue that $\mu([0,\infty))>\tfrac12$, a property we refer to as \emph{friendship paradox significance}. Brief sketches of the main ideas behind the proofs of these theorems are included. Full proofs can be found in Section \ref{sec:proofs2}.


\paragraph{Notation.} 

For sequences $(a_n)_{n\in\N}$ and $(b_n)_{n\in\N}$ of real numbers we write $a_{n} \sim b_{n}$ when $\lim_{n\to\infty} a_{n}/b_{n}=1$, $a_{n} \asymp b_{n}$ when $0 < \liminf_{n\to\infty} a_{n}/b_{n} \leq \limsup_{n\to \infty} a_{n}/b_{n} < \infty$, $a_{n} \lesssim b_{n}$ when $\limsup_{n\to\infty} \frac{a_{n}}{b_n} < \infty$, and $a_{n} \gtrsim b_{n}$ when $\liminf_{n\to\infty} \frac{a_{n}}{b_n} > 0$.


\section{Local convergence: main theorems}
\label{sec:locconv}

The asymptotic behaviour of the empirical distribution $\mu_{n}$ as $n\to\infty$ provides information on the friendship paradox for large complex networks. Theorem \ref{thm1} below shows that $\mu_{n}$  converges weakly to $\mu$ in probability for all locally tree-like random graphs, and quantify the friendship paradox.  

We start by setting up the notion of local convergence for random graphs from \cite[Chapter 2]{RvdH2}. Let $B_{r}^{(G)}(o)$ denote the rooted subgraph of $(G,o)$ in which all vertices have graph distance at most $r$ to $o$, i.e., if $\text{dist}_{G}$ denotes the graph distance in $G$, then 
\begin{equation*}
B_{r}^{(G)}(o) = \big((V(B_{r}^{(G)}(o)),E(B_{r}^{(G)}(o))),o\big)
\end{equation*}
with 
\begin{equation*}
\begin{aligned}
V(B_{r}^{(G)}(o)) &= \{i\in V(G)\colon\,\text{dist}_{G}(o,i)\leq r\},\\
E(B_{r}^{(G)}(o)) &= \{e\in E(G)\colon\,e=\{i,j\},\,\max\{\text{dist}_{G}(o,i),\text{dist}_{G}(o,j)\}\leq r\},
\end{aligned}
\end{equation*}
where for representing an edge we allow multisets. Write $G_1 \simeq G_2$ when $G_1$ and $G_2$ are \emph{isomorphic}. Let $\mathscr{G}$ be the set of all connected locally finite rooted graphs equipped with the metric
\begin{align*}
\mathrm{d}_{\mathscr{G}}\big((G_{1},o_{1}),(G_{2},o_{2})\big)
= \left(1+\sup\big\{r\geq 0\colon\,B_{r}^{(G_1)}(o_1) \simeq B_{r}^{(G_2)}(o_2)\big\}\right)^{-1},
\end{align*}
with the convention that two connected locally finite rooted graphs $(G_{1},o_{1})$ and $(G_{2},o_{2})$ are identified when $(G_{1},o_{1}) \simeq (G_{2},o_{2})$. 

\begin{definition}
\label{deflc}
{\rm Let $(G_{n})_{n\in\N}$ be a sequence of finite random graphs. For $n\in\N$, let $\mathscr{C}(G_{n},U_{n})$ denote the connected component of a uniformly chosen vertex $U_{n} \in V(G_{n})$ in $G_{n}$, viewed as a rooted graph with root vertex $U_{n}$.
\begin{itemize}
\item[(a)] 
$G_{n}$ converges \emph{locally weakly} to $(G,o)\in \mathscr{G}$ with (deterministic) law $\tilde{\nu}$ if, for every bounded and continuous function $h\colon\,\mathscr{G}\to\R$,
\begin{align*}
\E_{\tilde{\nu}_{n}}\big[h(\mathscr{C}(G_{n} ,U_{n}))\big] \to \E_{\tilde{\nu}}\big[h((G,o))\big],
\end{align*}
where $\E_{\tilde{\nu}_{n}}$ is the expectation with respect to the random vertex $U_{n}$ and the random graph $G_{n}$ with joint law $\tilde{\nu}_{n}$, while $\E_{\tilde{\nu}}$ is the expectation with respect to $(G, o)$ with law $\tilde{\nu}$.
\item[(b)] 
$G_{n}$ converges \emph{locally in probability} to $(G,o)\in \mathscr{G}$ with (possibly random) law $\tilde{\nu}$ if, for every bounded and continuous function $h\colon\,\mathscr{G}\to\R$,
\begin{align*}
\E_{\tilde{\nu}_{n}}\big[h(\mathscr{C}(G_{n} ,U_{n})) \mid G_{n}\big] \prob\E_{\tilde{\nu}}\big[h((G,o))\big].
\end{align*}
\end{itemize}
}\hfill$\spadesuit$
\end{definition}

\begin{theorem}
\label{thm1}
If $(G_{n})_{n\in\N}$ converges locally in probability to the almost surely locally finite, connected rooted random graph $(G,o)$, then $\mu_{n} \weak \mu$ as $n\to\infty$ in probability. In particular, as $n\to\infty$,
\begin{align*}
\mu_{n} (A) \prob \mu (A) \quad \quad \text{and} \quad \quad
\tilde{\mu}_{n} (A) \to \mu (A) \qquad \forall\,A \in \mathcal{B}(\R).
\end{align*}
\end{theorem}

\begin{theorem}
\label{thm2}
If $(G_{n})_{n\in\N}$ converges locally weakly to the almost surely locally finite, connected rooted random graph $(G,o)$, then $\tilde{\mu}_{n} \weak \mu$ as $n\to\infty$. In particular, as $n\to\infty$,
\begin{align*}
\tilde{\mu}_{n} (A) \to \mu (A) \qquad \forall\,A \in \mathcal{B}(\R).
\end{align*}
\end{theorem}

\begin{remark}
{\rm (a) For the special case where $(G,o)$ is an almost surely locally finite rooted random \emph{tree} $(G_{\infty},\phi)$, the friendship-bias of the root $\phi$ is
\begin{align}
\label{Dphidef}
\Delta_{\phi} = \left[\dfrac{1}{d_{\phi}} \sum_{j=1}^{d_{\phi}} (d_{j}+1)-d_{\phi}\right] 
\,\mathbbm{1}_{\{d_{\phi}\neq 0\}},
\end{align}
where $d_{\phi}$ is the degree of $\phi$, and $d_{j}$ is the size of the offspring of neighbour $j \in [d_{\phi}] = \{1,\ldots,d_{\phi}\}$ of $\phi$ (labelled in an arbitrary ordering). We write $\mu\colon\,\mathcal{B}(\R) \to [0,1]$ to denote the law of $\Delta_{\phi}$, and $\bar{\mu}$ to denote the law of $(G_{\infty},\phi)$. All examples treated in Section~\ref{sec:ex} have almost surely locally finite rooted random trees as local limits.\\ 
(b) For the special case where $(G_{\infty},\phi)$ is a rooted \emph{Galton-Watson tree} in which each individual in the population independently gives birth to a random number of children in $\N_{0}$ according to a common law $\nu$, the measure $\mu$ equals
\begin{align}
\label{mu}
\mu(\cdot) = \nu(0) \delta_{0}(\cdot) + \sum_{k\in\N} \nu(k)\nu^{\circledast k}\big(k(\cdot+k-1)\big),
\end{align}
where $\nu(k) = \nu(\{k\})$ is the probability that an individual has $k$ children, and $\nu^{\circledast k}$ is the $k$-fold convolution of $\nu$.
}\hfill$\spadesuit$
\end{remark}

We already know from Section~\ref{FPintro} that $\int_{\R} x \mu_n(\dd x) \geq 0$, and from Theorem \ref{thm1} that for all locally tree-like random graphs,
\begin{align*}
\dfrac{1}{n}\sum_{i \in [n]} \delta_{\Delta_{i,n}}([0,\infty))\prob \mu ([0,\infty)),\quad\quad n\to\infty.
\end{align*}
In other words, the proportion of vertices with non-negative friendship bias converges in probability to the probability that the friendship bias of the root of the limiting rooted tree is non-negative. If $\mu([0,\infty)) \geq\tfrac12$, then for large $n$ at least half of the vertices in $G_n$ have a degree that is larger than or equal to that of their neighbours. We therefore divide the friendship paradox for locally tree-like random graphs into two classes:

\begin{definition}
{\rm We say that the friendship paradox is \emph{significant} when $\mu([0,\infty)) \geq \tfrac12$ and \emph{insignificant} when $\mu([0,\infty)) < \tfrac12$. We say that the friendship paradox is \emph{marginally significant} when $\mu([0,\infty))= \tfrac{1}{2}$.}\hfill$\spadesuit$
\end{definition}

In Section~\ref{sec:ex} we analyse $\mu$ for four classes of sparse random graphs: the homogeneous Erd\H{o}s-R\'enyi random graph (Section~\ref{sec:HER}), the inhomogeneous Erd\H{o}s-R\'enyi random graph (Section~\ref{sec:IER}), the configuration model (Section~\ref{sec:CM}), and the preferential attachment model (Section~\ref{sec:PAM}). We prove that the friendship paradox is significant for specific one-parameter classes of configuration models and preferential attachment models. Furthermore, we demonstrate numerically that the friendship paradox is significant also for the homogeneous Erd\H{o}s-R\'enyi random graph and the inhomogeneous Erd\H{o}s-R\'enyi random graph. 

As part of our analysis of significance we propose the following conjecture:

\begin{conjecture}
\label{conjecture1}
Let $(X_i)_{i\in\N_0}$ be i.i.d.\ Binomial, Geometric or Poisson random variables, defined on a probability space $(\Omega, \mathcal{F},\P)$. Then 
\begin{equation*}
\P\bigg\{\sum_{i=1}^{X_{0}}X_{i}\geq X_{0}(X_{0}-1)\bigg\} \geq \frac{1}{2}
\end{equation*}
with the empty sum equal to zero.
\end{conjecture} 

\noindent
Conjecture \ref{conjecture1} says that if $d_{\phi},d_{1},d_{2},\ldots$ are i.i.d.\ Binomial, Geometric or Poisson random variables, then the friendship paradox is significant (recall \eqref{Dphidef}). We believe that this conjecture also holds for many other distributions with support in $\N_0$ (and that the i.i.d.\ assumption can be relaxed), but it does not hold for general discrete distributions with support in $\N_0$. For example, if $(X_i)_{i\in\N_0}$ are i.i.d.\ copies of 
\begin{align*}
X=\left\{
\begin{array}{ll}
1, &\text{with probability } \tfrac{1}{5},\\[0.2cm]
20, &\text{with probability } \tfrac{4}{5},
\end{array} 
\right.
\end{align*}
then $\P\{\sum_{i=1}^{X_{0}} X_{i}\geq X_{0}(X_{0}-1)\} = \tfrac{1}{5}+6\big(\tfrac{4}{5}\big)^{21}<\tfrac{1}{2}$. Moreover, the inequality is trivial for the Bernoulli distribution, while for the $\mathrm{Poisson}(\lambda)$ distribution it holds for sufficiently small values of $\lambda$, as demonstrated by the estimate 
\begin{align*}
\P\left\{\sum_{i=1}^{X_{0}}X_{i}\geq X_{0}(X_{0}-1)\right\} 
&\geq \E\Big[\Big(\frac{1}{2}+\frac{1}{2}\,\P\big\{X_{1}=X_{0}\big\}
+\frac{1}{2}\,\P\big\{X_{1}=X_{0}-1\big\}\Big)^{X_{0}}\Big] \\
&\geq\E\Big[\Big(\frac{1+\mathrm{e}^{-\lambda}}{2}\Big)^{X_{0}}\Big]
=\mathrm{e}^{-\frac{1}{2}\lambda(1-\mathrm{e}^{-\lambda})},
\end{align*}
where $\E$ denotes expectation with respect to $\P$.


\section{Proof of the local convergence theorems}
\label{sec:proofs}

In this section we provide the proofs of Theorems~\ref{thm1}--\ref{thm2}. The proofs are an application of the continuous mapping theorem in the setting of local convergence introduced in Definition~\ref{deflc}.


\subsection{Local convergence}
\label{sec:conv}


\subsubsection{Convergence locally in probability}

\begin{proof}[Proof of Theorem \ref{thm1}]
We replace $(G,o)$ by a connected locally finite rooted random graph $(G^{\prime},o)$ defined by
\begin{align*}
(G^{\prime}(\omega),o) = 
\left\{
\begin{array}{ll}
(G(\omega),o), &\text{if }(G(\omega),o) \text{ is locally finite}, \\
(H,o), &\text{otherwise},
\end{array} \right.
\end{align*}
where $(H,o)$ is an arbitrary connected locally finite deterministic graph rooted at $o$. Without loss of generality, we simplify the notation by assuming that $(G,o) \equiv (G^{\prime},o)$. Since $G_{n}$ converges locally in probability to $(G,o)$, for every bounded and continuous function $h\colon\,(\mathscr{G},\mathrm{d}_{\mathscr{G}}) \to (\R,\vert\cdot\vert)$ we have 
\begin{equation}
\label{a-1}
\E_{\bar{\mu}_{n}}\big[h(\mathscr{C}(G_{n},U_{n})) \mid G_{n}\big] \prob \E_{\bar{\mu}}\big[h((G,o))\big],
\end{equation}
where the expectation in the left-hand side is the conditional expected value of $h(\mathscr{C}(G_{n},U_{n}))$ given $G_{n}$ when $(G_{n},U_{n})$ has the joint law $\bar{\mu}_{n}$, and $\E_{\bar{\mu}}$ denotes expectation with respect to $(G,o)$ with a law $\bar{\mu}$. Since
\begin{align*}
\E_{\bar{\mu}_{n}}\big[h(\mathscr{C}(G_{n},U_{n})) \mid G_{n}\big] = \dfrac{1}{n} \sum_{i \in [n]} h(\mathscr{C}(G_{n},i)),
\end{align*}
where $\mathscr{C}(G_{n},i)$ is the connected component of $i$ in $G_{n}$ viewed as a rooted graph with root vertex $i$, it follows from \eqref{a-1} that
\begin{equation}
\label{a-2}
\dfrac{1}{n} \sum_{i \in [n]} h(\mathscr{C}(G_{n},i)) \prob \E_{\bar{\mu}}\big[h((G,o))\big]
\end{equation}
for every bounded and continuous function $h\colon\,(\mathscr{G},\mathrm{d}_{\mathscr{G}}) \to (\R ,\vert\cdot\vert)$.

Now, let $f\colon\,(\R ,\vert\cdot\vert) \to (\R ,\vert\cdot\vert)$ be an arbitrary bounded and continuous function. Define the function $h\colon\,(\mathscr{G},\mathrm{d}_{\mathscr{G}}) \to (\R,\vert\cdot\vert)$ by setting  
\begin{equation}
\label{h}
h((G,o)) = (f\circ g)((G,o)),
\end{equation}
where $g\colon\,(\mathscr{G},\mathrm{d}_{\mathscr{G}}) \to (\R,\vert\cdot\vert)$ is defined as $g((G,o))=\Delta_{o}^{(G)}$. Note that the boundedness of $f$ implies the boundedness of $h$. Also, for any $(G_1,o_1),(G_2,o_2)\in\mathscr{G}$, if $ \mathrm{d}_{\mathscr{G}}((G_1,o_1),(G_2,o_2))<\frac{1}{3}$, then 
\begin{equation*}
\sup\big\{r\geq 0\colon\,B_{r}^{(G_1)}(o_1) \simeq B_{r}^{(G_2)}(o_2)\big\}>2,
\end{equation*}
implying that $B_{r}^{(G_1)}(o_1) \simeq B_{r}^{(G_2)}(o_2)$ for all $r\in\{0,1,2\}$. Hence $g((G_1,o_1))=g((G_2,o_2))$, which means $g$ is continuous. Therefore $h$, as a composition of two continuous functions, is also continuous. But $h(\mathscr{C}(G_{n},i)) = f(\Delta_{i,n})$, and $h((G,o)) = f(\Delta_{o}^{(G)})$ a.s. Inserting this into \eqref{a-2}, we get
\begin{equation}
\label{a-3}
\dfrac{1}{n} \sum_{i \in [n]} f(\Delta_{i,n}) \prob \E_{\bar{\mu}}\big[f(\Delta_{o}^{(G)})\big].
\end{equation}
On the other hand, $\int_\R f\dd\mu_{n} = \frac{1}{n} \sum_{i \in [n]} f(\Delta_{i,n})$ and $\int_\R f\dd\mu = \E_{\bar{\mu}}[f(\Delta_{o}^{(G)})]$. Hence the convergence in \eqref{a-3} implies that
\begin{align*}
\int_\R f\dd\mu_{n} \prob \int_\R f\dd\mu,
\end{align*}
which settles the claim that $\mu_{n} \weak \mu$ as $n\to\infty$ in probability.

Finally, let $A\in\mathcal{B}(\R)$, and consider the function $h\colon\,(\mathscr{G},\mathrm{d}_{\mathscr{G}}) \to (\R,\vert\cdot\vert)$ defined by 
\begin{align}
\label{hindicator}
h((G,o))=\mathbbm{1}_{\{\Delta_{o}^{(G)}\in A\}}.
\end{align}
For any $(G_1,o_1),(G_2,o_2)\in\mathscr{G}$, if $ \mathrm{d}_{\mathscr{G}}((G_1,o_1),(G_2,o_2))<\frac{1}{3}$, then $B_{r}^{(G_1)}(o_1) \simeq B_{r}^{(G_2)}(o_2)$ for $r\in\{0,1,2\}$, implying that $h((G_1,o_1))=h((G_2,o_2))$. Hence, $h$ is continuous, and from \eqref{a-2} we have $\mu_{n}(A)\prob\mu(A)$. Applying the dominated convergence theorem, we also get that $\tilde{\mu}_{n}(A)\to\mu(A)$.
\end{proof}


\subsubsection{Convergence locally weakly}

\begin{proof}[Proof of Theorem \ref{thm2}]
We use the same notation as in the proof of Theorem \ref{thm1}. For $n\in\N$,
\begin{align*}
\tilde{\mu}_{n}(\cdot) = \dfrac{1}{n} \sum_{i \in [n]} \E_{n}[\delta_{\Delta_{i,n}}(\cdot)] 
= \dfrac{1}{n} \sum_{i \in [n]} \mu_{i,n}(\cdot),
\end{align*}
where $\mu_{i,n}(\cdot)$ is the law of $\Delta_{i,n}$. Let $f\colon\,(\R,\vert\cdot\vert) \to (\R,\vert\cdot\vert)$ be an arbitrary bounded and continuous function. We have
\begin{align*}
\int_\R f \dd\tilde{\mu}_{n} = \dfrac{1}{n} \sum_{i \in [n]} \int_\R f \dd\mu_{i,n}
=\dfrac{1}{n}\sum_{i \in [n]} \E_{n}[f(\Delta_{i,n})] = \E_{n}\left[\frac{1}{n}\sum_{i \in [n]}f(\Delta_{i,n})\right].
\end{align*}

Since $G_{n}$ converges locally weakly to $(G,o)$, for every bounded and continuous function $h\colon\,(\mathscr{G},\mathrm{d}_{\mathscr{G}}) \to (\R ,\vert\cdot\vert)$ we have 
\begin{align}\label{weaknew1}
\E_{\bar{\mu}_{n}}\big[h(\mathscr{C}(G_{n},U_{n}))\big] \to \E_{\bar{\mu}}\big[h((G,o))\big],
\end{align}
where the expectation in the left-hand side is with respect to $(G_{n},U_{n})$ with law $\bar{\mu}_{n}$. Using the fact that 
\begin{equation} \label{weaknew2}
\E_{\bar{\mu}_{n}}\big[h(\mathscr{C}(G_{n},U_{n}))\big] = \E_{n}\left[\E_{\bar{\mu}_{n}}\big[h(\mathscr{C}(G_{n},U_{n}))\mid G_{n}\big]\right]
= \E_{n}\left[\dfrac{1}{n}\sum_{i \in [n]}h(\mathscr{C}(G_{n},i))\right]  
\end{equation}
and taking $h$ as in \eqref{h}, we get
\begin{align*}
\int_\R f \dd\tilde{\mu}_{n} = \E_{n}\left[\frac{1}{n}\sum_{i \in [n]} f(\Delta_{i,n})\right] 
\to \E_{\bar{\mu}}\big[f(\Delta_{o}^{(G)})\big] = \int_\R f\dd \mu,
\end{align*}
which settles the claim that $\tilde{\mu}_{n} \weak \mu$.

Finally, let $A\in\mathcal{B}(\R)$. By substituting the bounded and continuous function $h$ defined in \eqref{hindicator} into \eqref{weaknew1}, and using \eqref{weaknew2}, we conclude that $\tilde{\mu}_{n}(A)\to\mu(A)$.
\end{proof}


\section{Four classes of sparse random graphs: main theorems}
\label{sec:ex}

In this section we focus on the computation of $\mu([0,\infty))$, $\E_{\bar{\mu}}[\Delta_{\phi}]$ and $\E_{\bar{\mu}}[\Delta_{\phi}^{2}]$, where $\E_{\bar{\mu}}$ denotes expectation with respect to $(G_{\infty},\phi)$ with law $\bar{\mu}$. We write $\P_{\bar{\mu}}$ to denote the probability measure of the probability space on which $(G_{\infty},\phi)$ is defined. 


\subsection{Homogeneous Erd\H{o}s-R\'enyi random graph}
\label{sec:HER}

Fix $\lambda \in (0,\infty)$. For $n\in\N$, let $\HER_{n}(\frac{\lambda}{n} \wedge 1)$ be the random graph on $[n]$ where each pair of vertices is independently connected by an edge with probability $\frac{\lambda}{n} \wedge 1$. This random graph converges locally in probability to a \emph{Galton-Watson tree with a Poisson offspring distribution with mean} $\lambda$ \cite[Theorem 2.18]{RvdH2}. 

\begin{theorem}
\label{thm1HER}
For every $\lambda \in (0,\infty)$, 
\begin{equation*}
\mu([0,\infty)) =  \sum_{k\in\N_{0}} \frac{\ee^{-\lambda}\lambda^{k}}{k!} 
\sum_{l\geq k(k-1)} \dfrac{\ee^{-\lambda k}(\lambda k)^{l}}{l!}.
\end{equation*} 
In particular, $\lim_{\lambda \downarrow 0} \mu([0,\infty)) = 1$ and $\lim_{\lambda \to \infty} \mu([0,\infty)) = \frac{1}{2}$, i.e., the friendship paradox is significant for both small and large $\lambda$, and becomes marginally significant in the limit as $\lambda\to\infty$.
\end{theorem}

\begin{theorem}
\label{thm2HER}
For every $\lambda \in (0,\infty)$,
\begin{align*}
&\E_{\bar{\mu}}\big[\Delta_{\phi}\big] = 1-(\lambda+1)\,\ee ^{-\lambda},\\
&\E_{\bar{\mu}}\big[\Delta_{\phi}^{2}\big] = \lambda \Big(\sum_{k\in\N}
k^{-1}\dfrac{\ee^{-\lambda}\lambda^{k}}{k!}+1\Big)-(\lambda+1)^{2}\,\ee ^{-\lambda}+1.
\end{align*}
In particular, 
\begin{align*}
\lim_{\lambda \downarrow 0} \E_{\bar{\mu}}[\Delta_{\phi}] = 0, &
\quad \lim_{\lambda \downarrow 0} \E_{\bar{\mu}}[\Delta_{\phi}^2] = 0,\\
\lim_{\lambda \to \infty} \E_{\bar{\mu}}[\Delta_{\phi}]=1, 
&\quad \lim_{\lambda \to \infty} \E_{\bar{\mu}}[\Delta_{\phi}^{2}] = \infty.
\end{align*}
\end{theorem}

\begin{theorem}
\label{thm3HER}
For every $\lambda \in (0,\infty)$,
\begin{align*}
\P_{\bar{\mu}}\{\Delta_{\phi}\geq x\}\sim \P_{\bar{\mu}}\big\{d_{\phi}=1 \big\} \P_{\bar{\mu}}\big\{d_{1}\geq x\big\},
\qquad x \to \infty.
\end{align*}
In particular,
\begin{align*}
\P_{\bar{\mu}}\{\Delta_{\phi}\geq x\} \sim \frac{1}{\sqrt{2\pi x}}\, \lambda\,\ee^{-2\lambda}
\exp\Big\{-x\log\Big(\frac{x}{\lambda\ee}\Big)\Big\}, \qquad x\to\infty.
\end{align*}
\end{theorem}

Theorems~\ref{thm1HER}--\ref{thm3HER} show that for the homogeneous Erd\H{o}s-R\'enyi random graph, which is a simple graph, the properties of $\mu$ are explicitly computable. In addition, as seen in Figure \ref{fig2}, numerical computations indicate that $ \mu([0,\infty))$ is a strictly decreasing function of $\lambda$, which leads us to \emph{conjecture} that $\mu$ is significant for all $\lambda \in (0,\infty)$ (in support of Conjecture~\ref{conjecture1}).

\begin{figure}[htbp]
\centering
\includegraphics[scale=0.4]{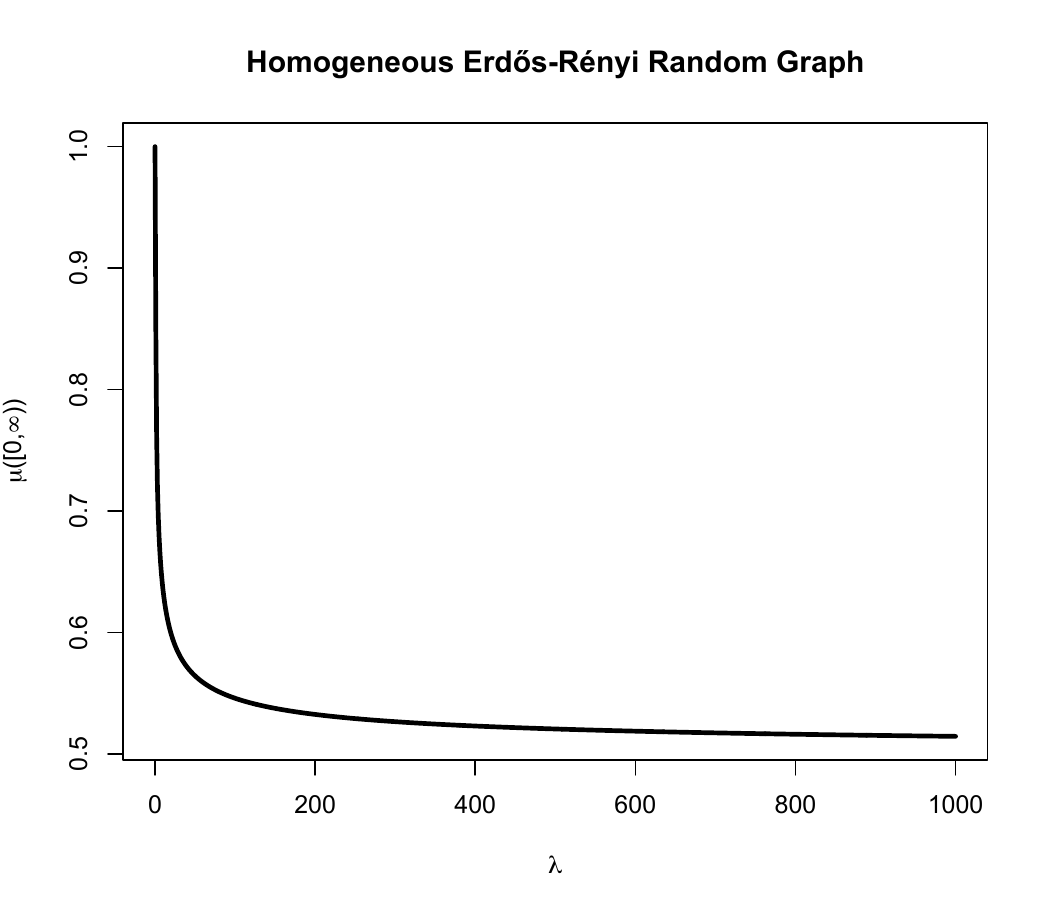}
\caption{\label{fig2} \textsl{Numerical plot of the map $\lambda\mapsto \sum_{k=0}^{10^{4}} \tfrac{\ee^{-\lambda}\lambda^{k}}{k!} \sum_{l\geq k(k-1)} \tfrac{\ee^{-\lambda k}(\lambda k)^{l}}{l!}$ for $10^{-8} \leq \lambda \leq 10^3$ for $\HER_{n}(\frac{\lambda}{n} \wedge 1)$.}}
\end{figure}

\paragraph{Ideas behind the proofs.}
In Theorem \ref{thm1HER}, the expression for $\mu([0,\infty))$ is derived from the fact that the sizes of the offsprings of the vertices in the tree are i.i.d. $\mathrm{Poisson}(\lambda)$. Moreover, the first limit is obtained by noting that $\mu([0,\infty))$ is bounded from below by $\mathrm{e}^{-\lambda}$. The proof of the second limit is based on the central limit theorem, with the restriction $\{\Delta_{\phi}\geq 0\}$ on the event $\{\sqrt{\lambda}\,(\log\lambda)^{-1}< |d_{\phi}-\lambda|\leq \sqrt{\lambda}\,\log\lambda\}$ with $\lambda$ sufficiently large. The proof of Theorem \ref{thm2HER} is straightforward and follows from the fact that $d_{\phi},(d_{i})_{i\in\N}$ are i.i.d.\ $\mathrm{Poisson}(\lambda)$. Theorem \ref{thm3HER} is obtained by showing that, as $x \to \infty$, 
\[
\begin{aligned}
&\P_{\bar{\mu}}\left\{\sum_{j=1}^{k} d_{j}\geq kx+k(k-1)\right\} 
\sim \P_{\bar{\mu}}\left\{\sum_{j=1}^{k}d_{j} = kx+k(k-1)\right\}, 
\qquad k\in\N,\\
&\P_{\bar{\mu}}\left\{\sum_{j=1}^{k}d_{j} = kx+k(k-1)\right\} 
= o(\P_{\bar{\mu}}\big\{d_1=x \big\}),
\qquad k\in\N\setminus\{1\}.
\end{aligned}
\] 
For more details we refer to Section \ref{sec:proofs2}.

 
\subsection{Inhomogeneous Erd\H{o}s-R\'enyi random graph}
\label{sec:IER}

Let $\cF$ be the class of \emph{non-constant} Riemann integrable functions $f\colon\,[0,1] \to (0,\infty)$ satisfying
\begin{align*}
M_{-} = M_{-}(f) =\inf_{x\in[0,1]} f(x)>0, \qquad M_{+} = M_{+}(f) = \sup_{x\in[0,1]}f(x)<\infty.
\end{align*} 
Fix $\lambda \in (0,\infty)$. For $n\in\N$, let $\IER_{n}(\lambda f)$ be the random graph on $[n]$ where each pair of vertices  $i,j \in [n]$ is independently connected by an edge with probability $\frac{\lambda}{n} f(\frac{i}{n})f(\frac{j}{n}) \wedge 1$. This random graph converges locally in probability to a \emph{unimodular multi-type marked Galton-Watson tree} with the following properties \cite[Theorem 3.14]{RvdH2}:
\begin{itemize}
\item[(1)] 
The root has type $Q$ with
\begin{align*}
\rho \{Q\in A\}=\ell_1(A) \qquad \forall\, A \subseteq [0,1],
\end{align*}
where $\rho$ is the probability distribution of the types of the vertices, and $\ell_1$ is the Lebesgue measure on $[0,1]$.
\item[(2)] 
Let $\beta_m = \int_{[0,1]} \dd y\,f^{m}(y)$. Any vertex other than the root takes an independent type $Q^{\prime}$ with  
\begin{align*}
\rho \{Q^{\prime}\in A\} = \beta_{1}^{-1} \int_{A} \dd y\,f(y) \qquad \forall\, A\subseteq [0,1].
\end{align*}
\item[(3)]
A vertex of type $x$ independently has offspring distribution $\mathrm{Poisson}(\lambda \beta_{1} f(x))$.
\end{itemize}
Property (1) says that $Q$ has a uniform $[0,1]$ distribution, while property (2) says that
\begin{equation*}
h_{Q^{\prime}}(y) = \beta_{1}^{-1} f(y)\mathbbm{1}_{\{y\in[0,1]\}} 
\end{equation*}
is the density function of $Q^{\prime}$. 

\begin{theorem}
\label{thm1IER}
Let $c_{k}(z)$ be the probability that a random variable with a distribution that is $\mathrm{Poisson}(\lambda\beta_{1}z)$ takes the value $k\in \N_0$.
\begin{itemize}
\item[(a)] 
For every $f\in\cF$ and $\lambda \in (0,\infty)$,
\begin{equation*}
\begin{aligned}
&\mu([0,\infty))\\ 
&= \sum_{k\in\N_{0}} \left(\int_{[0,1]} \dd x\,c_{k}(f(x))\right)
\left[\beta_{1}^{-k} \prod_{j=1}^{k} \int_{[0,1]} \dd x_{j} f(x_{j}) 
\sum_{l \geq k(k-1)} c_{l}\left(\sum_{j=1}^{k} f(x_{j})\right)\right],
\end{aligned}
\end{equation*}
where $\prod_{j=1}^{0} = 1$.
\item[(b)] 
For every $f\in\cF$,
\begin{align*}
&\lim_{\lambda \downarrow 0} \mu([0,\infty)) = 1,\\
&\lim_{\lambda \to \infty} \mu([0,\infty)) =\P_{\bar{\mu}}\{(Q,Q^{\prime})\in A_{f}\}
+\frac{1}{2}\P_{\bar{\mu}}\{(Q,Q^{\prime})\in B_{f}\},
\end{align*}
where
\begin{align*}
&A_{f} = \big\{(x,y)\in[0,1]\times[0,1]\colon\,f(x)<f(y)\big\},\\
&B_{f} = \big\{(x,y)\in[0,1]\times[0,1]\colon\,f(x)=f(y)\big\}.
\end{align*}
\item[(c)] 
For every $f\in\cF$, $\lim_{\lambda \to \infty} \mu([0,\infty)) \geq \frac{1}{2}$, which together with (a) implies that the friendship paradox is significant for both small and large $\lambda$.
\item[(d)]
If $\ell_{2}$ is the Lebesgue measure on $[0,1]\times[0,1]$ and $\ell_{2}(A_{f})>0$, then $\lim_{\lambda \to \infty}$ $\mu([0,\infty)) > \frac{1}{2}$. Moreover, if $\ell_{2}(A_{f})=0$, then $\lim_{\lambda \to \infty} \mu([0,\infty)) = \frac{1}{2}$, i.e., the friendship paradox is marginally significant in the limit as $\lambda \to\infty$.
\item[(e)]
If $f\in\cF$ is injective, then $\lim_{\lambda \to \infty} \mu([0,\infty)) = \P_{\bar{\mu}}\{(Q,Q^{\prime})\in A_{f}\}$.
\end{itemize}
\end{theorem}

\begin{theorem}
\label{thm2IER}
For every $f\in\cF$ and $\lambda\in(0,\infty)$,
\begin{equation}
\label{e1}
\E_{\bar{\mu}}\big[\Delta_{\phi}\big] = 1 - (\lambda\beta_{2}+1)\int_{[0,1]} \dd x\,\ee^{-\lambda \beta_{1} f(x)}
+ \lambda(\beta_{2}-\beta_{1}^{2}) 
\end{equation}
and 
\begin{align*}
&\E_{\bar{\mu}}\big[\Delta_{\phi}^{2}\big]
= \lambda\big(\beta_{2}+\lambda(\beta_{1}\beta_{3} -\beta_{2}^{2})\big)
\int_{[0,1]}\dd x \sum_{k\in\N}k^{-1}\dfrac{\ee^{-\lambda \beta_{1}f(x)}(\lambda \beta_{1}f(x))^{k}}{k!}\\
&\qquad\quad - \big(\lambda^{2}\beta_{2}^{2}+2\lambda\beta_{2}+1\big) \int_{[0,1]} \dd x\, \ee^{-\lambda \beta_{1} f(x)}
+ \lambda^{2}(\beta_{2}^{2}-\beta_{2}\beta_{1}^{2})+\lambda(2\beta_{2}-\beta_{1}^{2})+1.
\end{align*}
In particular,
\begin{align*}
\lim_{\lambda \downarrow 0} \E_{\bar{\mu}}\big[\Delta_{\phi}\big] = 0, 
&\quad \lim_{\lambda \downarrow 0} \E_{\bar{\mu}}\big[\Delta_{\phi}^2\big] = 0,\\
\lim_{\lambda\to\infty} \E_{\bar{\mu}}\big[\Delta_{\phi}\big] = \infty, 
&\quad \lim_{\lambda\to\infty} \E_{\bar{\mu}}\big[\Delta_{\phi}^{2}\big] = \infty.
\end{align*}
\end{theorem}

\begin{theorem}
\label{thm3IER}
For every $f\in\cF$ and $\lambda \in (0,\infty)$,
\begin{align*}
\P_{\bar{\mu}}\{\Delta_{\phi}\geq x\} \sim \P_{\bar{\mu}}\big\{d_{\phi}=1 \big\} \P_{\bar{\mu}}\big\{d_{1}\geq x\big\},
\qquad x \to \infty.
\end{align*}
In particular,
\begin{align*}
\P_{\bar{\mu}}\{\Delta_{\phi}\geq x\} \asymp \frac{1}{\sqrt{x}}\,\beta_{x}\,
\ee^{-x\log(\frac{x}{\lambda\beta_{1}\ee})}, \qquad x \to \infty,
\end{align*}
where $\beta_{x}=\int_{[0,1]} \dd y f^x(y)$. The asymptotic behaviour of $\beta_x$ can be derived in the following two cases: 
\begin{itemize}
\item[(a)] 
If $\ell_{1}$ is the Lebesgue measure on $[0,1]$ and $\ell_{1}(\{y\in [0,1]\colon\,f(y)=M_{+}\})>0$, then $\beta_{x} \asymp M_{+}^{x}$ as $x\to\infty$.
\item[(b)] 
If $f(y_{\star})=M_{+}$ and $f(y_{\star})-f(y) \asymp |y-y_{\star}|^\alpha$, $y \to y_{\star}$, for some $y_{\star} \in [0,1]$ and $\alpha \in (0,\infty)$, and $f$ is bounded away from $M_{+}$ outside any neighbourhood of $y_{\star}$, then $\beta_{x} \asymp x^{-1/\alpha}\,M_{+}^{x}$ as $x\to\infty$.
\end{itemize}
\end{theorem}

Theorems~\ref{thm1IER}--\ref{thm3IER} show that for the inhomogeneous Erd\H{o}s-R\'enyi random graph, which is a simple graph, the properties of $\mu$ are again explicitly computable. In addition, as seen in Figure \ref{fig3}, numerical computations indicate that $\mu([0,\infty))$ is a strictly decreasing function of $\lambda$, regardless which $f\in\cF$ is used, which leads us to \emph{conjecture} that $\mu$ is significant for all $f\in\cF$ and $\lambda \in (0,\infty)$ (in support of a possible extension of Conjecture~\ref{conjecture1} to a case where $X_{0}$ has a different distribution than the other random variables).

\begin{figure}[htbp]
\centering
\includegraphics[scale=0.5]{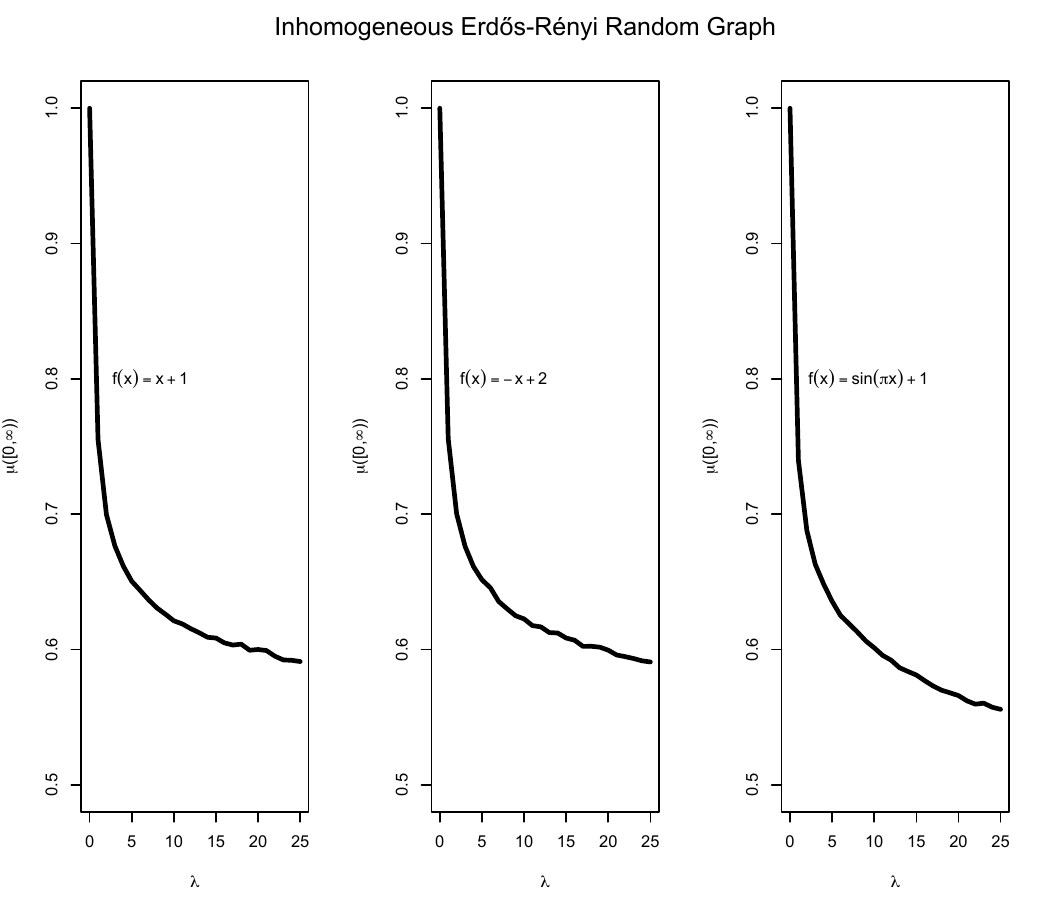}
\caption{\label{fig3} \textsl{Numerical plot of the map $\lambda\mapsto \mu([0,\infty))$ for $\IER_{n}(\lambda f)$ with an increasing, a decreasing and a non-monotonic function $f$, estimated with the help of Monte Carlo integration.}}
\end{figure}

In Theorem~\ref{thm2IER}, the term $\beta_{2}-\beta_{1}^{2}$ is the variance of $f(U)$ when $U$ is a uniform $[0,1]$ random variable. If $f$ would be a constant function, then the third term in \eqref{e1} would vanish, which would yield $\E_{\bar{\mu}}[\Delta_{\phi}] \to 1$ as $ \lambda\to\infty$. However, because $f$ is non-constant, the inhomogeneity may cause an individual to have friends who have significantly more friends than the individual does. For example, if $f(x) = a\mathbbm{1}_{A}(x) + b\mathbbm{1}_{B}(x)$ with $a<b$, $A \cap B=\emptyset$ and $A \cup B=[0,1]$, then $f$ divides the individuals into \emph{two communities} $A_{n} = nA \cap [n]$ and $B_{n} = nB \cap [n]$. Two individuals inside community $A_{n}$ are friends with probability $a^{2}\lambda/n$, while two individuals inside community $B_{n}$ are friends with probability $b^{2}\lambda/n$. The friends of individuals in community $A_{n}$ can live in community $B_{n}$, where individuals are more likely to have friends, which explains why $\E_{\bar{\mu}}[\Delta_{\phi}] \to \infty$ as $ \lambda\to\infty$.

\paragraph{Ideas behind the proofs.}
Theorem \ref{thm1IER}(a) follows from the definition of the probability function $\nu_{0}(\cdot) = \int_{[0,1]} \dd x\, c_{\cdot}(f(x))$ for the offspring distribution of the root, and the fact that the random variable $\sum_{j=1}^{k} d_{j}$, given that the type of the $j$-th child of the root is $x_{j}$ for each $j\in\{1,\ldots ,k\}$, has distribution $\text{Poisson}(\lambda\beta_{1}\sum_{j=1}^{k} f(x_{j}))$. The first limit of Theorem \ref{thm1IER}(b) is obtained by noting that $\mu([0,\infty))\ge \ee^{-\lambda\beta_{1} M_{+}}$, the second limit is obtained by proving the following statement using the central limit theorem and restricting to the event $\{|d_{\phi}-\lambda\beta_{1}f(x)|\leq\sqrt{\lambda\beta_{1}f(x)}\,\log \lambda\}$:
\[
\lim_{\lambda\rightarrow\infty}\P_{\bar{\mu}}\left\{\sum_{j=1}^{d_{\phi}}d_{j}\geq d_{\phi}(d_{\phi}-1),d_{\phi} 
\geq 1\mid (Q, Q^{\prime})=(x,y)\right\} =\mathbbm{1}_{\left\{(x,y)\in A_{f}\right\}}
+ \tfrac{1}{2}\mathbbm{1}_{\left\{(x,y)\in B_{f}\right\}}.
\] 
To establish the latter, a further restriction to the event $\{|d_{\phi}-\lambda\beta_{1}f(x)|>\sqrt{\lambda\beta_{1}f(x)}\,(\log \lambda)^{-1}\}$ is needed when $(x,y)\in B_{f}$. Theorems \ref{thm1IER}(c-d) follow by applying Fubini's theorem, which gives
\[
\Upsilon_{f}=\iint_{A_{f}}\dd x\,\dd y\,\beta_{1}^{-1}f(x) =1-\P_{\bar{\mu}}\big\{(Q,Q^{\prime})\in A_{f}\big\}
-\P_{\bar{\mu}}\big\{(Q,Q^{\prime})\in B_{f}\big\},
\] 
where $\P_{\bar{\mu}}\big\{(Q,Q^{\prime})\in A_{f}\big\}=0$ when $\ell_{2}(A_{f})=0$ and is larger than $\Upsilon_{f}$ when $\ell_{2}(A_{f})>0$. Theorem \ref{thm1IER}(e) follows from the fact that $\P_{\bar{\mu}}\{(Q,Q^{\prime})\in B_{f}\}=0$ when $f$ is injective. Theorem \ref{thm2IER} follows from the fact that $\E_{\bar{\mu}}[d_{\phi}]=\lambda\beta_{1}^{2}$, $\E_{\bar{\mu}}[d_{1}] = \lambda\beta_{2}$, $\E_{\bar{\mu}}[d_{\phi}^{2}] = \lambda\beta_{1}^{2}+\lambda^{2}\beta_{1}^{2}\beta_{2}$ and $\E_{\bar{\mu}}[d_{1}^{2}] = \lambda\beta_{2}+\lambda^{2}\beta_{1}\beta_{3}$. The steps neede to prove the main part of Theorem \ref{thm3IER} are similar to those for Theorem \ref{thm3HER}. The proof of part (a) is straightforward, while the proof of part (b) relies on the asymptotics
\[
\frac{\beta_{x}}{M_{+}^{x}} \asymp\int_{[0,1]} \dd y\,\exp\{-x\,\frac{|y_{\star}-y|^\alpha}{M_{+}}\} 
\asymp \int_{y_{\star}-x^{-1/\alpha}}^{y_{\star}+x^{-1/\alpha}} \dd y, \qquad x \to \infty. 
\]
For more details we refer to Section \ref{sec:proofs2}.


\subsection{Configuration model}
\label{sec:CM}

For $n\in\N$ and a sequence of deterministic non-negative integers $\mathbf{d}_{n}=(d_{1,n},\ldots ,d_{n,n})$, let $\CM_{n}(\mathbf{d}_{n})$ be the random graph on $[n]$ drawn uniformly at random from the set of all graphs with the same degree sequence $\mathbf{d}_{n}$.  Let $D_{n}$ be the degree of a uniformly chosen vertex in $[n]$. If $D_{n}$ converges in distribution to a random variable $D$ with $\P\{D>0\}=1$ such that $\E[D_{n}] =\frac{1}{n} \sum_{i \in [n]} d_{i,n} \to \E[D]<\infty$, then this random graph converges locally in probability to a \emph{unimodular branching process tree} with \emph{root offspring distribution} $p=(p_{k})_{k\geq 0}$ given by $p_{k} = \P\{D=k\}$ \cite[Theorem 4.1]{RvdH2} and with \emph{biased offspring distribution} $p^{\star}=(p^{\star}_{k})_{k\geq 0}$ given by $p^{\star}_{k} = (k+1)p_{k+1}/\E[D]$ for all the other vertices \cite[Definition 1.26]{RvdH2}. 

Note that $\CM_{n}(\mathbf{d}_{n})$ is a multi-graph, i.e., it possibly has self-loops and multiple edges. In many real-life applications, $D$ often exhibits a power-law distribution. Our first result shows that assuming $D$ to have a power-law distribution with a finite first moment implies the significance of the friendship paradox.

\begin{theorem}
\label{thm1CM}
Consider the special case where $p_{k} = k^{-\tau}/\zeta(\tau)$, $\tau>2$, with $\zeta$ the Riemann zeta function. Then $\mu([0,\infty)) > \frac{1}{2}$ for all $\tau>2$. In addition, $\lim_{\tau \downarrow 2} \mu([0,\infty)) = 1$ and $\lim_{\tau \to \infty} \mu([0,\infty)) = 1$.
\end{theorem}

The next result describes the first and second moment of $\Delta_{\phi}$ and for this we do not need to assume that $D$ has a power law distribution. 
\begin{theorem}
\label{thm2CM}
Let $D$ be such that $\P(D>0)=1$ and $\E[D]<\infty$. Then 
\begin{align*}
\E_{\bar{\mu}}\big[\Delta_{\phi}\big] = \dfrac{\mathbb{V}\mathrm{ar}(D)}{\E[D]}.
\end{align*}
For every $D$ satisfying the above conditions and $\E[D^2]<\infty$,
\begin{align*}
\E_{\bar{\mu}}\big[\Delta_{\phi}^{2}\big] = \dfrac{\left(\E[D^{3}]\E[D]-(\E[D^{2}])^{2}\right)
\E[D^{-1}]+\E[D^{2}]\mathbb{V}\mathrm{ar}(D)}{(\E[D])^{2}}.
\end{align*}
In particular, $\mathbb{V}\mathrm{ar}_{\bar{\mu}}(\Delta_{\phi})\geq \mathbb{V}\mathrm{ar}(D)$.
\end{theorem}

The next result derives the tail of the measure $\mu$ when $D$ follows a power law distribution.

\begin{theorem}
\label{thm3CM}
Consider the special case where $p_{k} = k^{-\tau}/\zeta(\tau)$, $\tau>2$. Then
\begin{align*}
\P_{\bar{\mu}}\{\Delta_{\phi}\geq x\} \asymp \P_{\bar{\mu}}\big\{d_{1}\geq x\big\}, \qquad x \to \infty.
\end{align*} 
In particular,
\begin{align*}
\P_{\bar{\mu}}\{\Delta_{\phi}\geq x\} \asymp x^{-(\tau-2)}, \qquad x \to \infty.
\end{align*}
\end{theorem}

Theorems~\ref{thm1CM}--\ref{thm3CM} show that for the configuration model, which is a multi-graph, the properties of $\mu$ are somewhat harder to come by than for the homogeneous or inhomogeneous Erd\H{o}s-R\'enyi Random Graph. Fig.~\ref{fig4} shows numerical upper and lower bound on $\mu([0,\infty))$ for the special case where $p_{k} = k^{-\tau}/\zeta(\tau)$, $\tau>2$. The black curves are given by the following maps:
\begin{equation*}
\begin{array}{ll}
&\tau\mapsto \sum_{k=1}^{200} p_{k}\P_{\bar{\mu}}\big\{d_{1}\geq k(k-1)\big\},\\[0.2cm]
&\tau\mapsto 1\wedge \sum_{k=1}^{200} kp_{k}\P_{\bar{\mu}}\big\{d_{1}\geq k-1\big\}.
\end{array}
\end{equation*}
Fig.~\ref{fig4} shows that $\mu([0,\infty))$ is not a monotone function of $\tau$, and is always larger than $\tfrac45$. It achieves its minimum value inside the interval $[2,4]$.

\begin{figure}[htbp]
\centering
\includegraphics[scale=0.4]{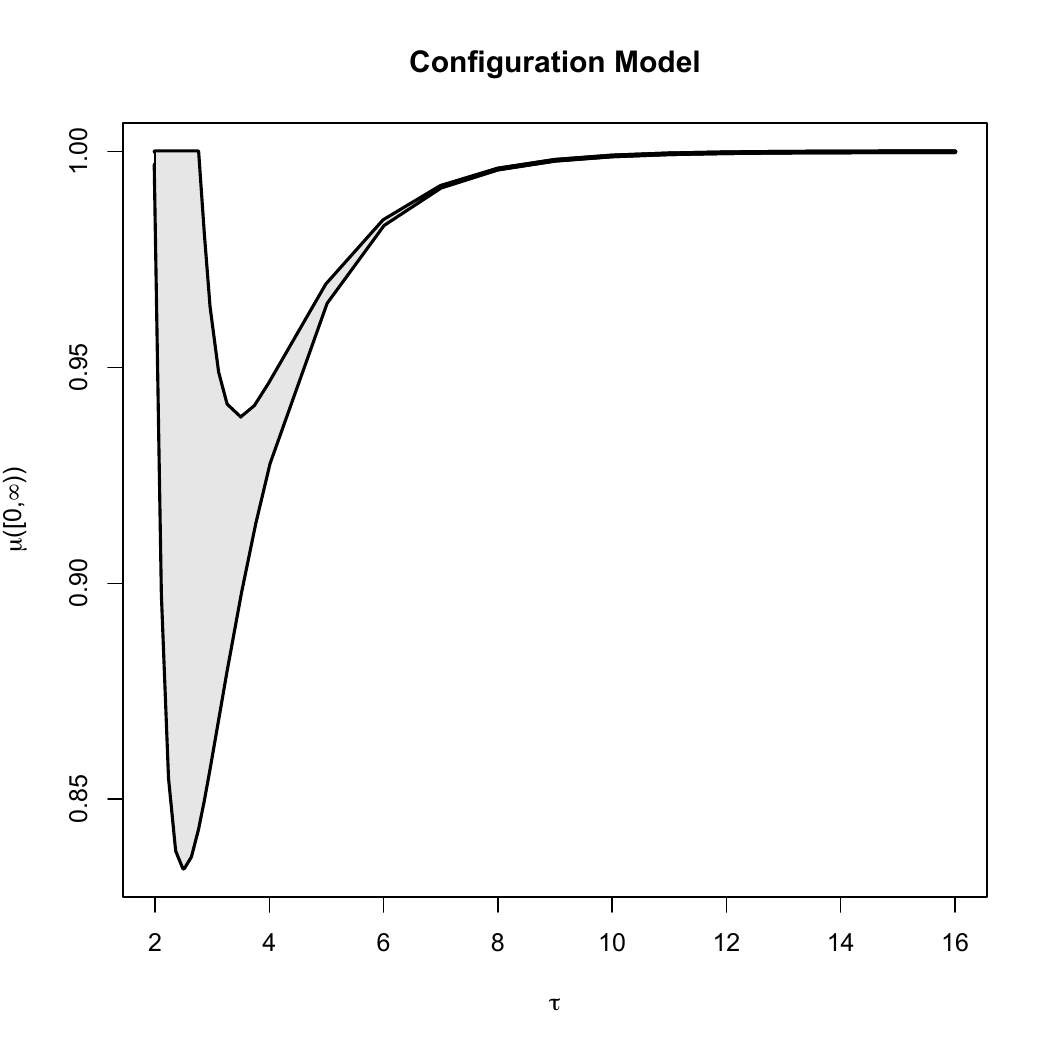}
\caption{\label{fig4} \textsl{Numerical plot of the region containing the map $\tau\mapsto \mu([0,\infty))$ for $\CM_{n}(\mathbf{d}_{n})$ with $p_{k} = k^{-\tau}/\zeta(\tau)$.}}
\end{figure}

The following example shows that for the configuration model with a bimodal degree distribution the friendship paradox is significant as soon as one of the two degrees is $1$.

\begin{example}
{\rm Consider the case where the limit is a tree with $p_{k} = p\mathbbm{1}_{\{M_{1}\}}(k)+(1-p)\mathbbm{1}_{\{M_{2}\}}(k)$ for some $M_{1},M_{2}\in\N$, $M_{1}\leq M_{2}$ and $p\in(0,1)$. Then
\begin{align*}
p_{k}^{\star} = \dfrac{M_{1}p}{M_{1}p+M_{2}(1-p)}\mathbbm{1}_{\{M_{1}-1\}}(k)
+\dfrac{M_{2}(1-p)}{M_{1}p+M_{2}(1-p)}\mathbbm{1}_{\{M_{2}-1\}}(k).
\end{align*}
If $M_{1}=M_{2}$, then $p_{k} =\mathbbm{1}_{\{M_{1}\}}(k)$, $p_{k}^{\star}=\mathbbm{1}_{\{M_{1}-1\}}(k)$ and $\mu([0,\infty))=1$. If $M_{1}<M_{2}$, then 
\begin{align}
\label{bimodal}
\mu([0,\infty))= p+(1-p)\bigg(\dfrac{M_{2}(1-p)}{M_{1}p+M_{2}(1-p)}\bigg)^{M_{2}},
\end{align}
which is $>\frac{1}{2}$ when $p\in[\tfrac{1}{2},1)$. If $M_{1}=1<M_{2}$, $p\in(0,\frac{1}{2})$. Suppose $\mu([0,\infty))\leq\frac{1}{2}$, then \eqref{bimodal} implies
\begin{align*}
M_{2}\log\Big(1+\frac{p}{M_{2}(1-p)}\Big)\geq \log\Big(\frac{1-p}{\tfrac{1}{2}-p}\Big),
\end{align*}
which (because $\log(1+x)\leq x$, $x\geq 0$) in turn implies that
\begin{align*}
\log\Big(\frac{1-p}{\tfrac{1}{2}-p}\Big)-\frac{p}{1-p}\leq 0,
\end{align*}
which is a contradiction. Consequently, if $M_{1}=1<M_{2}$ and $p\in(0,1)$, then $\mu([0,\infty))>\frac{1}{2}$.}
\hfill$\spadesuit$ 
\end{example}

The following example shows that for the configuration model the friendship paradox is not always significant.

\begin{example}
{\rm Let $\mathbf{d}_{n}$, $n\in \N$, be such that $\lfloor\tfrac{n}{10}\rfloor$ degrees are $9$ and the other degrees are $10$. It is easily checked that $\CM_{n}(\mathbf{d}_{n})$ converges locally in probability to a tree with root offspring distribution $p_{k} = \tfrac{1}{10}\mathbbm{1}_{\{9\}}(k)+\tfrac{9}{10}\mathbbm{1}_{\{10\}}(k)$, and that $\mu([0,\infty))=\tfrac{1}{10}+\tfrac{9}{10}\big(\tfrac{10}{11}\big)^{10} < \frac{1}{2}$. The latter follows from \eqref{bimodal} with $p=\tfrac{1}{10}$, $M_{1}=9$ and $M_{2}=10$.}\hfill$\spadesuit$
\end{example}

\paragraph{Ideas behind the proofs.} Since $\mu([0,\infty)) \geq \max\{\zeta(3\tau - 4)/((\tau-1) \zeta(\tau)), 1/\zeta(\tau)\}$, Theorem \ref{thm1CM} follows. The proof of Theorem \ref{thm2CM} follows from $\E_{\bar{\mu}}[d_{1}] = (\E[D^{2}]-\E[D])/\E[D]$ when $\E[D]<\infty$, and $\E_{\bar{\mu}}[d_{1}^{2}] = (\E[D^{3}]+\E[D]-2\E[D^{2}])/\E[D]$ when $\E[D^2]<\infty$ as well. In the latter case, the Cauchy-Schwarz inequality gives $\E_{\bar{\mu}}[\Delta_{\phi}^{2}] \geq \E[D^{2}]\mathbb{V}\mathrm{ar}(D)/(\E[D])^{2}$. Theorem \ref{thm3CM} follows by showing that 
\[
\frac{(x+2)^{-\tau +2}}{(\tau-2)\zeta(\tau -1)} \leq \P_{\bar{\mu}}\{d_{1}\geq x\}\leq \frac{x^{-\tau +2}}{(\tau -2)\E[D]},
\qquad x>0.
\] 
For more details we refer to Section \ref{sec:proofs2}.


\subsection{Preferential attachment model}
\label{sec:PAM}

Imagine a social network with new individuals arriving one by one, expanding the social network by one vertex at each arrival. The new individual makes connections with the other individuals by becoming acquainted with them. A realistic assumption is that the new individual is more likely to become acquainted with individuals who already have a large number of acquaintances, which is sometimes called the \emph{rich-get-richer} phenomenon (\cite[Chapter 1]{RvdH2}).
The preferential attachment model describes networks that grow over time, where newly added vertices are more likely to connect to old vertices with higher degrees than to old vertices with lower degrees. It is a graph sequence denoted by $(\PAM_{n}^{(m,\delta)})_{n\in\N}$, depending on two parameters $m\in\N$ and $\delta \geq -m$, such that at each time $n$ there are $n$ vertices and $mn$ edges, i.e., the total degree is $2mn$. For the friendship paradox, it is natural to assume $m=1$. For $m>1$: at time $1$ there is one vertex that establishes $m$ friendship relations with itself, while at time $2$ the new vertex either establishes $m$ friendship relations with itself or $m$ friendship relations with the previous vertex. So the friendship is higher when $m>1$.  Henceforth we restrict ourselves to the more natural case where $m=1$, so that one vertex with one edge is added per unit time. The following brief overview is taken from \cite[Chapter 8]{RvdH1}.

Fix $\delta \geq -1$. In this case, $\PAM_{1}^{(1,\delta)}$ consists of a single vertex $v_{1}$ with a single self-loop. For $n\in\N$, suppose that $v_{1},\ldots,v_{n}$ are the vertices of $\PAM_{n}^{(1,\delta)}$ with degrees $d_{1,n},\ldots,d_{n,n}$, respectively. Then, given $\PAM_{n}^{(1,\delta)}$, the rule for obtaining $\PAM_{n+1}^{(1,\delta)}$ is such that a single vertex $v_{n+1}$ with a single edge attached to a vertex in $\{v_{1},\ldots,v_{n+1}\}$ is added according to the following probabilities:
\begin{align*}
\P \big\{v_{n+1} \to  v_{i}~\big|~\PAM_{n}^{(1,\delta)}\big\} 
= \left\{
\begin{array}{ll}
\dfrac{d_{i,n}+\delta}{n(2+\delta)+(1+\delta)}, &\text{if } i \in [n],\\
\dfrac{1+\delta}{n(2+\delta)+(1+\delta)}, &\text{if } i = n+1.
\end{array} 
\right.
\end{align*}
The role of the parameter $\delta$ is that of a shift: for $\delta = 0$ the attachment probabilities are linear in the degrees, for $\delta \neq 0$ affine in the degrees. The $\delta$ also allows some flexibility in terms of the degree distribution. The power law exponent of the degree distribution is given by $3+\delta/m$. 

For $\delta > -1$, $\PAM_{n}^{(1,\delta)}$ converges locally in probability to a multi-type discrete-time branching process called the \textit{P\'{o}lya point tree} \cite[Theorem 5.26]{RvdH2} (see also \cite{BBCS}).\footnote{It is worth noting here that $(\PAM_{n}^{(1,\delta)})_{n\in\N}$ is denoted by $(\PAM_{n}^{(1,\delta)}(a))_{n\in\N}$ in \cite{RvdH2}.} In this tree, each vertex $u$ is labelled by finite words as $u=u^{1}u^{2} \cdots u^{l}$, as in the Ulam-Harris way of labelling trees and $\phi$ represents the empty word or the root of the tree. Then, $u=\phi$ carries a type $A_{u}$ and a vertex $u\neq\phi$ carries a type $(A_{u},L_{u})$, where $A_{u} \in [0,1]$ specifies the age of $u$ and  $L_{u}$ a label in $\{\mathsf{y},\mathsf{o}\}$ that indicates whether $u$ is younger than its parent by denoting $L_{u}=\mathsf{y}$ or older than its parent by denoting $L_{u}=\mathsf{o}$. An old child is actually the older neighbour to which the initial edge of the parent is connected, whereas younger children are younger vertices that use an edge to connect to their parent. The number of children with label $\mathsf{o}$ of a vertex is either $0$ or $1$ and defined deterministically, but the number of children with label $\mathsf{y}$ of a vertex is random. This rooted tree is defined inductively as follows. 

First, the root $\phi$ takes an age $A_{\phi}$ that is chosen uniformly at random from $[0,1]$, but it does not take any label in $\{\mathsf{y},\mathsf{o}\}$ because it has no parent. Then, by recursion, the remainder of the tree is constructed in such a way that if the vertex $u$ has the Ulam-Harris word $u^{1}u^{2} \cdots u^{l}$ and the type $(A_{u},L_{u})$ (or $A_{u}$ in the case that $u=\phi$), then $u_{j}=uj= u^{1}u^{2} \cdots u^{l}j$ is defined as follows:
\begin{itemize}
\item[(1)] 
If either $L_{u}=\mathsf{o}$ or $u=\phi $, then $u$ has $\no(u) = 1$ offspring with label $ \mathsf{o} $. Otherwise it has $\no(u) = 0$ offspring with label $\mathsf{o}$. In the case $\no(u) = 1$, set $u_{1}$ as the offspring of $u$ with label $\mathsf{o}$. This offspring has an age $A_{u_{1}}$ given by  
\begin{equation*}
A_{u_{1}} = U_{u_{1}}^{\frac{2+\delta}{1+\delta}} A_{u},
\end{equation*}
where $U_{u_{1}}$ is a uniform $[0,1]$ random variable that is independent of everything else.
\item[(2)] 
Let $A_{u_{\no(u)+1}},\ldots,A_{u_{\no(u)+\ny(u)}}$ be the (ordered) points of a Poisson point process on $[A_{u},1]$ with intensity 
\begin{equation*}
\rho_{u}(x) = \dfrac{\Gamma_{u}}{2+\delta}\dfrac{x^{-(1+\delta)/(2+\delta)}}{A_{u}^{1/(2+\delta)}},
\end{equation*}
where $\Gamma_{u}$ is a random variable selected independently of everything else as
\begin{align*}
\Gamma_{u} \overset{d}=\left\{
\begin{array}{ll}
\text{Gamma}(1+\delta,1), &\text{if } u \text{ is the root or of label } \mathsf{y}, \\
\text{Gamma}(2+\delta ,1), &\text{if } u\text{ is of label } \mathsf{o}.
\end{array} 
\right.
\end{align*}
Then the vertices $u_{\no(u)+1},\ldots,u_{\no(u)+\ny(u)}$ are defined as the children of $u$ with label $\mathsf{y}$.
\end{itemize}

Note that $(\PAM_{n}^{(1,\delta)})_{n\in\N}$ is a multi-graph, i.e., it possibly has self-loops and multiple edges. 

\begin{theorem}
\label{thm1PAM}
$\mu([0,\infty)) > \tfrac{1}{2}$. In addition, $\lim_{\delta \downarrow -1} \mu([0,\infty)) = 1$. 
\end{theorem}

\begin{theorem}
\label{thm2PAM}
Abbreviate $p_{\delta} = \E_{\bar{\mu}}[d_{\phi}^{-1}] = \sum_{k\in\N} \frac{(2+\delta)\Gamma(3+2\delta)\Gamma(k+\delta)}{k\Gamma(1+\delta)\Gamma(k+3+2\delta)}$. Then
\begin{align*}
\E_{\bar{\mu}}\big[\Delta_{\phi}\big]
&\left\{
\begin{array}{ll}
\in \tfrac{2+\delta}{\delta}(\frac{1}{2}+p_{\delta}) + \big[-(1-p_{\delta}),0\big], 
&\text{if } \delta\in (0,\infty), \\
=\infty, 
&\text{if } \delta\in (-1,0],
\end{array} \right.
\end{align*}
and $\E_{\bar{\mu}}[\Delta_{\phi}^{2}]<\infty$ if and only if $ \delta \in (1,\infty)$. 
\end{theorem}

\begin{theorem}
\label{thm3PAM}
For every $\delta > -1$,
\begin{align*}
\P_{\bar{\mu}}\{\Delta_{\phi}\geq x\} \asymp \P_{\bar{\mu}}\big\{d_{1}\geq x\big\}, \qquad x \to \infty.
\end{align*} 
In particular,
\begin{equation*}
\P_{\bar{\mu}}\{\Delta_{\phi}\geq x\} \asymp x^{-(\tau -2)}, \qquad x \to \infty,
\end{equation*}
with $\tau=\tau(\delta) =3+\delta$.
\end{theorem}

Figure \ref{fig5} shows a numerical plot of a lower bound on $\delta\mapsto\mu([0,\infty))$ for $(\PAM_{n}^{(1,\delta)})_{n\in\N}$,
given by
\begin{align*}
\sum_{k=1}^{50} \P_{\bar{\mu}}\big\{d_{\phi}= k\big\}\P_{\bar{\mu}}\big\{d_{1}\geq k(k-1)\mid d_{\phi}= k\big\},
\end{align*}
in which we have used Monte Carlo integration to estimate $\P_{\bar{\mu}}\{d_{1}\geq k(k-1)\mid d_{\phi}= k\}$. Fig.~\ref{fig5} shows that $\mu([0,\infty))$ may not be a monotone function of $\delta$, and is always greater than $0.99$. In addition to confirming that $\lim_{\delta \downarrow -1} \mu([0,\infty)) = 1$, the figure leads us to \emph{conjecture} that $\lim_{\delta \to \infty} \mu([0,\infty)) = 1$. Since the preferential attachment model captures self-reinforcement in friendship networks, these strong forms of the friendship paradox are perhaps plausible. 

\begin{figure}[htbp]
\centering
\includegraphics[scale=0.4]{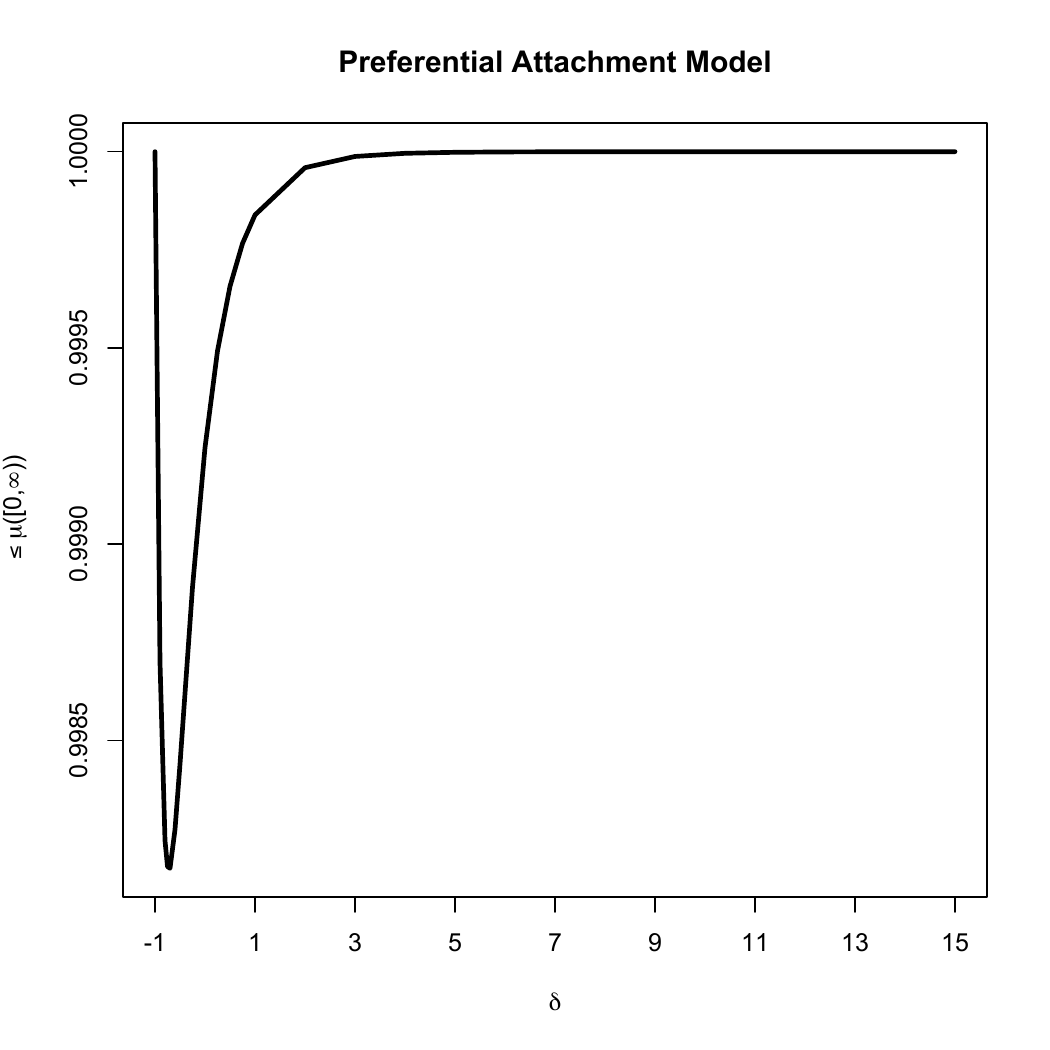}
\caption{\label{fig5} \textsl{Numerical plot of a lower bound on $\delta\mapsto \mu([0,\infty))$ for $(\PAM_{n}^{(1,\delta)})_{n\in\N}$.}}
\end{figure}

\paragraph{Ideas behind the proofs.} 
Theorem \ref{thm1PAM} follows from the inequality $\mu([0,\infty))\geq \P_{\bar{\mu}}\{d_{\phi}=1\}=\frac{2+\delta}{3+2\delta}$. Let $d_{i}$ denote the size of the offspring of $\phi_{i} = \phi i$. The first moment in Theorem \ref{thm2PAM} is obtained by noting that $\E_{\bar{\mu}}[d_{\phi}]=2$, that, for $k\in\N$, the conditional expectation $\E_{\bar{\mu}}[d_{1}\mid d_{\phi}=k]$ equals $1+\tfrac{2+\delta}{2\delta}(k+2)$ for $\delta>0$ and $\infty$ otherwise, and that $0 \leq \E_{\bar{\mu}}[\sum_{j=2}^{k} d_{j}\mid d_{\phi}=k] \leq k-1$ when $k\geq 2$. The main point about the second moment in Theorem \ref{thm2PAM} is that, for any $k\in\N$, the conditional expectation $\E_{\bar{\mu}}[d_{1}^{2}\mid d_{\phi}=k]$ is finite if and only if $\delta>1$. The lower asymptotic bound in Theorem \ref{thm3PAM} is derived by showing that, for $k$ and $x$ large enough,
\begin{align*}
&\P_{\bar{\mu}}\big\{d_{1}\geq x\mid d_{\phi}=k\big\}
\geq \gamma_{\delta}\, x^{-(1+\delta)}, \qquad
\P_{\bar{\mu}}\big\{d_{1}\geq k^{2}(x+1)\mid d_{\phi}=k\big\}
\ge \gamma_{\delta}\, (kx)^{-(1+\delta)}, 
\end{align*}
where $\gamma_{\delta}=\frac{2^{-(1+\delta)}}{\Gamma(3+\delta)\Gamma(4+2\delta)}$. The proof of the upper asymptotic bound  in Theorem \ref{thm3PAM} requires more intricate bounds. Namely, it can be shown that, for $k\in\N$ and $x$ large enough, 
\begin{align*}
\P_{\bar{\mu}}\big\{d_{1}\geq x\mid d_{\phi}=k\big\} \leq \gamma_{1,k,\delta}\,(x-1)^{-(1+\delta)},
\end{align*}
and, for $k\geq 2$,
\begin{align*}
\P_{\bar{\mu}}\Big\{\sum_{i=2}^{k}d_{i}\geq (k-1) x \mid d_{\phi}=k\Big\}
\leq \gamma_{2,k,\delta}\,x^{-(1+\delta)},
\end{align*}
where
\[
\gamma_{1,k,\delta}=\frac{\Gamma(k+3+2\delta)}{\Gamma(k+2+\delta)},
\qquad \gamma_{2,k,\delta}= \frac{\Gamma(k+3+2\delta)}{\Gamma(1+\delta)\Gamma(k)}2^{-k}.
\]
Because $\E_{\bar{\mu}}[\gamma_{1,d_{\phi},\delta}]<\infty$, the theorem follows from these upper bounds. For more details we refer to Section \ref{sec:proofs2}.


\section{Four classes of sparse random graphs: proof of the main theorems}
\label{sec:proofs2}
In this section we provide the full proofs of all the results discussed in Section \ref{sec:ex}.

\subsection{Homogeneous Erd\H{o}s-R\'{e}nyi random graph}


\paragraph{Proof of Theorem \ref{thm1HER}.}

Since $\nu = \mathrm{Poisson}(\lambda)$, it follows from \eqref{mu} that
\begin{align*}
&\mu([0,\infty))
= \nu(0) + \sum_{k\in\N} \nu(k)\nu^{\circledast k}\big([k(k-1),\infty)\big)\\
&= \ee^{-\lambda} + \sum_{k\in\N} \frac{\ee^{-\lambda}\lambda^{k}}{k!} 
\sum_{l\geq k(k-1)} \dfrac{\ee ^{-\lambda k}(\lambda k)^{l}}{l!}
= \sum_{k\in\N_{0}} \frac{\ee^{-\lambda}\lambda^{k}}{k!} \sum_{l\geq k(k-1)} 
\dfrac{\ee ^{-\lambda k}(\lambda k)^{l}}{l!}.
\end{align*}
Since
\begin{align*}
\mu([0,\infty)) \geq \ee^{-\lambda} + \lambda\ee^{-\lambda},
\end{align*}
it follows that $\lim_{\lambda\downarrow 0} \mu([0,\infty)) = 1$. 

We now show that $\lim_{\lambda\to\infty} \mu([0,\infty)) = 1/2$. Fix $\lambda>8$.  Define
\begin{align*}
Z_{\lambda}=\dfrac{d_{\phi}-\lambda}{\sqrt{\lambda}}, \qquad
Z_{\lambda}^{\prime}=\dfrac{\sum_{j=1}^{d_{\phi}}d_{j}-\lambda d_{\phi}}{\sqrt{\lambda d_{\phi}}}.
\end{align*}
Rewriting $\Delta_{\phi}$ in terms of $Z_{\lambda}$ and $Z_{\lambda}^\prime$ we have
\begin{align}
\label{az1}
\mu([0,\infty))
&=\ee^{-\lambda}+\P_{\bar{\mu}}\bigg\{\sum_{j=1}^{d_{\phi}}d_{j}\geq d_{\phi}(d_{\phi}-1),d_{\phi}\geq 1\bigg\}\nonumber\\
&=\ee^{-\lambda}+\P_{\bar{\mu}}\bigg\{Z_{\lambda}^{\prime}\sqrt{1+\tfrac{1}{\sqrt{\lambda}}
Z_{\lambda}}\geq \sqrt{\lambda}Z_{\lambda}\big(1-\tfrac{1}{\lambda}\big) -1+ Z_{\lambda}^{2}, \,\, d_{\phi}\geq 1\bigg\}.
\end{align}
We note that $\P_{\bar{\mu}}\{Z_{\lambda}\geq -\sqrt{\lambda}\}=1$ because $d_{\phi}$ is a non-negative random variable. Define the event $A_{\lambda}=\{|Z_{\lambda}|\leq \log\lambda,d_{\phi}\geq 1\}$. Since, by the continuous mapping theorem and central limit theorem, $\tfrac{|Z_{\lambda}|}{\log\lambda}\to 0$ in probability as $\lambda\to\infty$, we have
\begin{align}
\label{az2}
\P_{\bar{\mu}}\big\{A_{\lambda}^{c}\big\}\to 0, \qquad \lambda\to\infty.
\end{align}
On the other hand, 
\begin{align}
\label{az3}
&\P_{\bar{\mu}}\bigg\{ \Big\{ Z_{\lambda}^{\prime}\left(1+\tfrac{1}{\sqrt{\lambda}}Z_{\lambda}\right)^{1/2}
\geq \sqrt{\lambda}Z_{\lambda}\big(1-\tfrac{1}{\lambda}\big) -1+ Z_{\lambda}^{2}\Big\}\cap A_{\lambda} \bigg\}\nonumber\\
&=\P_{\bar{\mu}}\bigg\{ \Big\{ Z_{\lambda}^{\prime}\left(1+O(\tfrac{\log\lambda}{\sqrt{\lambda}})\right)^{1/2}
\geq \sqrt{\lambda} Z_{\lambda}\big(1-\tfrac{1}{\lambda}\big)-1+O(\log^{2}\lambda)\Big\}\cap A_{\lambda} \bigg\}.
\end{align}
Now define
\begin{align*}
B_{1,\lambda} = \Big\{Z_{\lambda}< -\frac{1}{\log\lambda}\Big\}, \quad
B_{2,\lambda} = \Big\{Z_{\lambda}> \frac{1}{\log\lambda}\Big\}, \quad
B_{3,\lambda} = \Big\{|Z_{\lambda}|\leq \frac{1}{\log\lambda}\Big\}.
\end{align*}
By the continuous mapping theorem and the central limit theorem, there exists a standard Gaussian random variable $N$ defined on a probability space $(\Omega ,\mathcal{F},\P)$ such that $Z_{\lambda}+\tfrac{1}{\log\lambda}\Rightarrow N$, $Z_{\lambda}-\tfrac{1}{\log\lambda}\Rightarrow N$ and $|Z_{\lambda}|-\tfrac{1}{\log\lambda}\Rightarrow |N|$ as $\lambda\rightarrow\infty$. Therefore as $\lambda\to\infty$
\begin{align}
&\P_{\bar{\mu}}\big\{B_{1,\lambda}\big\} \to \P\big\{N<0 \big\} =\tfrac{1}{2},\label{az4-1}\\
&\P_{\bar{\mu}}\big\{B_{2,\lambda}\big\} \to  \P\big\{N>0 \big\} =\tfrac{1}{2},  \nonumber\\
&\P_{\bar{\mu}}\big\{B_{3,\lambda}\big\} \to  \P\big\{|N|\leq0 \big\} =0.
\label{az4-2}
\end{align}
Again, using the continuous mapping theorem and the central limit theorem, we have
\begin{align*}
\tilde Z_{\lambda}\,\frac{\log\lambda}{\sqrt{\lambda}}
+O\big(\tfrac{\log^{3}\lambda}{\sqrt{\lambda}}\big) \Longrightarrow 0, \qquad \lambda\to\infty,
\end{align*}
where we abbreviate $\tilde Z_{\lambda} = Z_{\lambda}^{\prime}(1+O(\tfrac{\log\lambda}{\sqrt{\lambda}}))^{\tfrac{1}{2}}$.
Note that $g(\lambda)=\tfrac{\log\lambda}{\sqrt{\lambda}}$ is a strictly decreasing function on the interval $(8,\infty)$, which implies that $\tfrac{\log\lambda}{\sqrt{\lambda}}<\tfrac{\log 8}{\sqrt{8}}<\tfrac{7}{8}$. Hence
\begin{align}
\label{az5}
&\P_{\bar{\mu}}\bigg\{ \Big\{ \tilde Z_{\lambda}
\geq \sqrt{\lambda} Z_{\lambda}\big(1-\tfrac{1}{\lambda}\big)-1+O(\log^{2}\lambda)\Big\}
\cap B_{2,\lambda} \bigg\}\nonumber\\
&\leq \P_{\bar{\mu}}\bigg\{\tilde Z_{\lambda}\,\frac{\log\lambda}{\sqrt{\lambda}}
+O\big(\tfrac{\log^{3}\lambda}{\sqrt{\lambda}}\big)>\tfrac{7}{8}-\tfrac{\log 8}{\sqrt{8}}\bigg\}
\to 0, \qquad \lambda\to\infty.
\end{align}
Moreover, 
\begin{align}
\label{az6}
&\P_{\bar{\mu}}\bigg\{ \Big\{\tilde Z_{\lambda}
\geq \sqrt{\lambda} Z_{\lambda}\big(1-\tfrac{1}{\lambda}\big)-1+O(\log^{2}\lambda)\Big\}
\cap A_{\lambda} \cap B_{1,\lambda}\bigg\}\nonumber\\
&= \P_{\bar{\mu}}\big\{A_{\lambda} \cap B_{1,\lambda} \big\}
-\P_{\bar{\mu}}\bigg\{ \Big\{\tilde Z_{\lambda}
< \sqrt{\lambda} Z_{\lambda}\big(1-\tfrac{1}{\lambda}\big)-1+O(\log^{2}\lambda)\Big\}
\cap A_{\lambda} \cap B_{1,\lambda} \bigg\}\nonumber\\
&\to \tfrac{1}{2},\qquad \lambda\to\infty,
\end{align}
because it follows from \eqref{az2} and \eqref{az4-1} that
\begin{align*}
\P_{\bar{\mu}}\big\{A_{\lambda} \cap B_{1,\lambda} \big\}
=\P_{\bar{\mu}}\big\{B_{1,\lambda} \big\}-\P_{\bar{\mu}}\big\{A^{c}_{\lambda} \cap B_{1,\lambda} \big\}
\to \tfrac{1}{2}, \qquad \lambda\to\infty,
\end{align*}
and
\begin{align*}
&\P_{\bar{\mu}}\bigg\{ \Big\{\tilde Z_{\lambda}
< \sqrt{\lambda} Z_{\lambda}\big(1-\tfrac{1}{\lambda}\big)-1+O(\log^{2}\lambda)\Big\}
\cap A_{\lambda} \cap B_{1,\lambda} \bigg\}\\
&\leq \P_{\bar{\mu}}\bigg\{\tilde Z_{\lambda}\,\tfrac{\log\lambda}{\sqrt{\lambda}}
+O\big(\tfrac{\log^{3}\lambda}{\sqrt{\lambda}}\big)<-\tfrac{7}{8}\bigg\}\to 0, \qquad \lambda\to\infty.
\end{align*}
Equations \eqref{az1}--\eqref{az3} and \eqref{az4-2}--\eqref{az6} imply that $\lim_{\lambda\to \infty} \mu([0,\infty)) = \tfrac{1}{2}$. 


\paragraph{Proof of Theorem \ref{thm2HER}.}

Fix $\lambda \in (0,\infty)$. Since $(d_{j})_{j\in\N}$ are i.i.d.\ copies of $d_{\phi}$,
\begin{align*}
\E_{\bar{\mu}}\big[\Delta_{\phi}\big] 
&= \sum_{k\in\N}\nu(k)\E_{\bar{\mu}}\big[\Delta_{\phi}\mid d_{\phi}=k\big]
=\sum_{k\in\N}\nu(k)\E_{\bar{\mu}}\bigg[\frac{1}{k}\sum_{j=1}^{k}d_{j}+1-k\bigg]\\
&=\E_{\bar{\mu}}\big[d_{\phi}\big]\big(1-\nu(0)\big)+1-\nu(0)-\E_{\bar{\mu}}\big[d_{\phi}\big]
=1-(\lambda+1)\,\ee ^{-\lambda},
\end{align*}
which tends to $0$ as $\lambda\downarrow 0$ and $1$ as $\lambda\to\infty$. Similarly,
\begin{align*}
\E_{\bar{\mu}}\big[\Delta_{\phi}^{2}\big] 
&= \sum_{k\in\N}\nu(k)\E_{\bar{\mu}}\big[\Delta_{\phi}^{2}\mid d_{\phi}=k\big]
=\sum_{k\in\N}\nu(k)\E_{\bar{\mu}}\bigg[\Big(\frac{1}{k}\sum_{j=1}^{k}d_{j}+1-k\Big)^{2}\bigg]\\
&=\sum_{k\in\N}\nu(k)\E_{\bar{\mu}}\bigg[\frac{1}{k^{2}}\sum_{j=1}^{k}d_{j}^{2}
+\frac{1}{k^{2}}\sum_{i=1}^{k}\sum_{j=1 \atop {j\neq i}}^{k}d_{i}d_{j}+1+k^{2}
+\dfrac{2}{k}\sum_{j=1}^{k}d_{j}-2\sum_{j=1}^{k}d_{j}-2k\bigg]\\
&=\mathbb{V}\mathrm{ar}_{\bar{\mu}}\big(d_{\phi}\big)\E_{\bar{\mu}}\Big[d_{\phi}^{-1}
\mathbbm{1}_{\{d_{\phi}\neq 0\}}\Big]+\big(\E_{\bar{\mu}}\big[d_{\phi}\big]\big)^{2}\big(1-\nu(0)\big)+1-\nu(0)\\
&\quad\quad+\E_{\bar{\mu}}\big[d_{\phi}^{2}\big]+2\E_{\bar{\mu}}\big[d_{\phi}\big]\big(1-\nu(0)\big)
-2\big(\E_{\bar{\mu}}\big[d_{\phi}\big]\big)^{2}-2\E_{\bar{\mu}}\big[d_{\phi}\big]\\
&= \lambda \Big(\sum_{k\in\N}k^{-1}\dfrac{\ee^{-\lambda}\lambda^{k}}{k!}+1\Big)-(\lambda+1)^{2}\,\ee ^{-\lambda}+1.
\end{align*}
Since
\begin{align*}
\sum_{k\in\N}k^{-1}\dfrac{\ee^{-\lambda}\lambda^{k}}{k!}
=\E_{\bar{\mu}}\Big[d_{\phi}^{-1}\mathbbm{1}_{\{d_{\phi}\neq 0\}}\Big]\in [0,1],
\end{align*}
it follows that $\lim_{\lambda\downarrow 0} \E_{\bar{\mu}}[\Delta_{\phi}^{2}]=0$ and $\lim_{\lambda \to \infty} \E_{\bar{\mu}}[\Delta_{\phi}^{2}] = \infty$.


\paragraph{Proof of Theorem \ref{thm3HER}.}

Fix $\lambda \in (0,\infty)$. First, we show that
for each $k\in\N$,
\begin{align}
\label{tazi17}
\P_{\bar{\mu}}\Big\{\sum_{j=1}^{k} d_{j}\geq kx+k(k-1)\Big\}\sim\P_{\bar{\mu}}\Big\{\sum_{j=1}^{k}d_{j}= kx+k(k-1)\Big\},
\qquad x \to \infty.
\end{align}
Indeed, without loss of generality we may assume that $x\in\N$. Since, for each $k,l\in\N$, 
\begin{align*}
&\dfrac{\P_{\bar{\mu}}\big\{\sum_{j=1}^{k}d_{j}= kx+k(k-1)+l\big\}}
{\P_{\bar{\mu}}\big\{\sum_{j=1}^{k}d_{j}= kx+k(k-1)\big\}}\\
&=\dfrac{(\lambda k)^{l}}{(kx+k(k-1)+1)\times\cdots\times(kx+k(k-1)+l)} \to 0, \qquad x\to\infty,
\end{align*}
and
\begin{align*}
\sum_{l\in\N}\dfrac{\P_{\bar{\mu}}\big\{\sum_{j=1}^{k}d_{j}= kx+k(k-1)+l\big\}}
{\P_{\bar{\mu}}\big\{\sum_{j=1}^{k}d_{j}= kx+k(k-1)\big\}}
\leq\sum_{l\in\N}\dfrac{(\lambda k)^{l}}{l!}\leq\ee^{\lambda k}<\infty,
\end{align*}
\eqref{tazi17} follows by the dominated convergence theorem.

Next, we show that for $k\in\N\setminus\{1\}$,
\begin{equation}
\label{tazi18}
\P_{\bar{\mu}}\big\{\sum_{j=1}^{k}d_{j}= kx+k(k-1)\big\}=o\left(\P_{\bar{\mu}}\big\{d_1=x \big\}\right), \qquad x\to \infty.
\end{equation}
Indeed, for $k\in\N\setminus\{1\}$ we have, by Stirling's approximation,
\begin{align*}
&\dfrac{\P_{\bar{\mu}}\big\{\sum_{j=1}^{k}d_{j}= kx+k(k-1)\big\}}{\P_{\bar{\mu}}\big\{d_1=x \big\}}
=\dfrac{\P_{\bar{\mu}}\big\{\sum_{j=1}^{k}d_{j}= kx\big\}}{\P_{\bar{\mu}}\big\{d_1=x \big\}}
\dfrac{(\lambda k)^{k(k-1)}}{(kx+1)\times\cdots\times(kx+k(k-1))}\nonumber\\
&\sim\dfrac{\ee^{-\lambda(k-1)}}{\sqrt{k}}\left(\dfrac{x}{\lambda\ee}\right)^{-x(k-1)}
\dfrac{(\lambda k)^{k(k-1)}}{(kx+1)\times\cdots\times(kx+k(k-1))}
\to 0, \qquad x \to \infty.
\end{align*}
Observe that for $x\ge 1$,
\begin{align*}
\sum_{k\in\N}\dfrac{\P_{\bar{\mu}}\big\{d_{\phi}=k \big\}\P_{\bar{\mu}}
\big\{\sum_{j=1}^{k}d_{j}\geq kx+k(k-1)\big\}}{\P_{\bar{\mu}}\big\{d_1\geq x \big\}} \le 
\sum_{k\in\N} k\P_{\bar{\mu}}\big\{d_{\phi}=k \big\}= \E_{\bar{\mu}}\big[d_{\phi}\big] <\infty.
\end{align*}
Using \eqref{tazi17}, \eqref{tazi18} and the dominated convergence theorem, we have
\begin{align}
\label{forIHEproof}
&\lim_{x\to\infty}\dfrac{\P_{\bar{\mu}}\{\Delta_{\phi}\geq x\}}{\P_{\bar{\mu}}\big\{d_{\phi}=1 \big\} 
\P_{\bar{\mu}}\big\{d_{1}\geq x\big\}}\nonumber\\
&=\lim_{x\to\infty}\dfrac{\sum_{k\in\N}\P_{\bar{\mu}}\big\{d_{\phi}=k \big\}
\P_{\bar{\mu}}\big\{\sum_{j=1}^{k}d_{j}\geq kx+k(k-1)\big\}}{\P_{\bar{\mu}}\big\{d_{\phi}=1 \big\} 
\P_{\bar{\mu}}\big\{d_{1}\geq x\big\}}\nonumber\\
&=1+\sum_{k=2}^{\infty}\dfrac{\P_{\bar{\mu}}\big\{d_{\phi}=k \big\} }{\P_{\bar{\mu}}\big\{d_{\phi}=1 \big\}}
\lim_{x\to\infty}\dfrac{ \P_{\bar{\mu}}\big\{\sum_{j=1}^{k}d_{j}\geq kx+k(k-1)\big\}}
{\P_{\bar{\mu}}\big\{d_{1}\geq x\big\}}\nonumber\\
&=1+\sum_{k=2}^{\infty}\dfrac{\P_{\bar{\mu}}\big\{d_{\phi}=k \big\}}
{\P_{\bar{\mu}}\big\{d_{\phi}=1 \big\} }\lim_{x\to\infty}
\dfrac{\P_{\bar{\mu}}\big\{\sum_{j=1}^{k}d_{j}= kx+k(k-1)\big\}}{\P_{\bar{\mu}}\big\{d_{1}= x\big\}}
=1.
\end{align}
This, together with \eqref{tazi17} and Stirling's formula, yields
\begin{align*}
\P_{\bar{\mu}}\big\{\Delta_{\phi}\geq x\big\}
&\sim\P_{\bar{\mu}}\big\{d_{\phi}=1 \big\} \P_{\bar{\mu}}\big\{d_{1}\geq x\big\}
\sim\P_{\bar{\mu}}\big\{d_{\phi}=1 \big\} \P_{\bar{\mu}}\big\{d_{1}= x\big\}\\
&\sim \frac{1}{\sqrt{2\pi x}}\,\lambda\,\ee^{-2\lambda}\exp\Big\{-x\log\Big(\frac{x}{\lambda\ee}\Big)\Big\}, \qquad x \to \infty.
\end{align*}


\subsection{Inhomogeneous Erd\H{o}s-R\'{e}nyi random graph}


\paragraph{Proof of Theorem \ref{thm1IER}.}

(a) Let $c_{k}(z)$ be the probability that a random variable with a $\mathrm{Poisson}(\lambda \beta_{1} z)$-distribution takes the value $k$. If $\nu_{0}(k)$ denotes the probability that the root has $k$ children, then
\begin{align*}
\nu_{0}(k) = \int_{[0,1]} \dd x\, c_{k}(f(x)).
\end{align*}
Moreover, if $Q_{i}$ is the type of the $i$-th child of the root, then
\begin{align*}
\sum_{j=1}^{k} d_{j}~\big |~ (Q_{j})_{j=1}^k = (x_{j})_{j=1}^k \overset{d}=
\text{Poisson}\left(\lambda\beta_{1}\sum_{j=1}^{k} f(x_{j})\right),
\end{align*}
and if $\nu_{0}^{\circledast k}$ is the convolution of the probability distributions of $(d_{j})_{j=1}^k$, then
\begin{align*}
\nu_{0}^{\circledast k}\big([k(k-1),\infty)\big) 
= \beta_{1}^{-k}\prod_{j=1}^{k} \int_{[0,1]} \dd x_{j}\,f(x_{j}) \sum_{l\geq k(k-1)} c_{l}\left(\sum_{j=1}^{k} f(x_{j})\right).
\end{align*}
Hence
\begin{align*}
\mu([0,\infty)) = \nu_{0}(0) + \sum_{k\in\N}\nu_{0}(k)\nu_{0}^{\circledast k}\big([k(k-1),\infty)\big).
\end{align*}

\medskip\noindent
(b) The fact that $\mu([0,\infty))\ge\nu_{0}(0)\ge \ee^{-\lambda\beta_{1} M_{+}}$ implies $\lim_{\lambda\downarrow 0} \mu([0,\infty)) = 1$.

For large $\lambda$ we follow the strategy of using central limit theorem as in the proof of Theorem \ref{thm1HER}.  Let $\lambda>\max\{1,\lambda^{\prime}\}$ with $\lambda^{\prime}$ such that $\tfrac{\log\lambda}{\sqrt{\lambda}}<\frac{M_{-}}{2}$ for all $\lambda>\lambda^{\prime}$. For $x,y\in [0,1]$, define
\begin{align*}
Z_{\lambda}=Z_{\lambda}(x) = \dfrac{d_{\phi}-\lambda \beta_{1}f(x)}{\sqrt{\lambda \beta_{1}f(x)}}, \qquad
Z_{\lambda}^{\prime} = Z_{\lambda}^{\prime}(y) = \dfrac{\sum_{j=1}^{d_{\phi}}d_{j}
-\lambda \beta_{1}f(y)d_{\phi}}{\sqrt{\lambda \beta_{1}f(y)d_{\phi}}}.
\end{align*}
Also, put
\begin{align*}
&\varphi_{1}=\varphi_{1}(x,y)= \beta_{1}^{2}f(x)f(y),\\
&\varphi_{2}(\lambda)=\varphi_{2}(x,y,\lambda)
=\beta_{1}^{3/2}f^{1/2}(x)\big(2f(x)-f(y)\big)-\lambda^{-1}\beta_{1}^{1/2}f^{1/2}(x),\\
&\varphi_{3}(\lambda)=\varphi_{3}(x,y,\lambda)=\beta_{1}^{2}f(x)\big(f(x)-f(y)\big)-\lambda^{-1}\beta_{1}f(x),\\
&\varphi_{4}=\varphi_{4}(x)=\beta_{1}f(x).
\end{align*}
We have
\begin{align}
\label{tazi0}
&\mu([0,\infty))\nonumber\\
&=\nu_{0}(0) +\int_{[0,1]} \dd x \int_{[0,1]} \dd y\,\beta_{1}^{-1} f(y)\,
\P_{\bar{\mu}}\bigg\{ \sum_{j=1}^{d_{\phi}}d_{j}\geq d_{\phi}(d_{\phi}-1),d_{\phi} 
\geq 1~\Big |~ Q=x, Q^{\prime}=y \bigg\}
\end{align}
with
\begin{align}
\label{tazi00}
\nu_{0}(0)=\int_{[0,1]} \dd x\,\ee^{-\lambda\beta_{1}f(x)}
\leq \ee^{-\lambda M_{-}^{2}}\to 0, \qquad \lambda\to\infty.
\end{align}
Let $\P_{\bar{\mu}, x,y}\{\cdot\} = \P_{\bar{\mu}}\left\{\cdot\mid Q=x, Q^{\prime}=y\right\}$. Rewriting $\Delta_{\phi}$ in terms of $Z_{\lambda}$, $Z_{\lambda}^\prime$ and $\{\varphi_j\}_{j=1}^4$, we have
\begin{align}
\label{tazi1}
&\P_{\bar{\mu}, x, y}\bigg\{ \sum_{j=1}^{d_{\phi}}d_{j}\geq d_{\phi}(d_{\phi}-1),d_{\phi}
\geq 1~\bigg\}\nonumber\\
&=\P_{\bar{\mu}, x,y}\bigg\{  Z^{\prime}_{\lambda}\left( \varphi_{1}+O\big(\tfrac{1}{\sqrt{\lambda}}\big) 
Z_{\lambda}\right)^{1/2} \geq \sqrt{\lambda}\varphi_{2}(\lambda)Z_{\lambda}+\lambda\varphi_{3}(\lambda)
+\varphi_{4}Z_{\lambda}^{2},\, d_{\phi}\geq 1~\bigg\}.
\end{align}
Similarly to the proof of Theorem \ref{thm1HER}, taking the event $A_{\lambda}=\{|Z_{\lambda}|\leq \log\lambda,d_{\phi}\geq 1\}$, we have
\begin{align}
\label{tazi2}
\P_{\bar{\mu}}\big\{A_{\lambda}^{c}\mid Q=x\big\}\to 0, \qquad \lambda\to\infty .
\end{align}
Equations \eqref{tazi1}--\eqref{tazi2} imply that
\small
\begin{align}
\label{tazi3}
&\lim_{\lambda\to\infty}\P_{\bar{\mu}, x,y}\bigg\{ \sum_{j=1}^{d_{\phi}}d_{j}
\geq d_{\phi}(d_{\phi}-1),d_{\phi}\geq 1~ \bigg\}\nonumber\\
&=\lim_{\lambda\to\infty}\P_{\bar{\mu}, x, y}\bigg\{ \Big\{ Z^{\prime}_{\lambda}
\left(\varphi_{1}+O\big(\tfrac{1}{\sqrt{\lambda}}\big) Z_{\lambda}\right)^{1/2} \geq \sqrt{\lambda}\varphi_{2}(\lambda)Z_{\lambda}
+\lambda\varphi_{3}(\lambda)+\varphi_{4}Z_{\lambda}^{2}\Big\}\cap A_{\lambda}~\bigg\}\nonumber\\
&= \lim_{\lambda\to\infty}\P_{\bar{\mu},x,y}\bigg\{\Big\{Z^{\prime}_{\lambda}\left(\varphi_{1}
+O\big(\tfrac{\log\lambda}{\sqrt{\lambda}}\big) \right)^{1/2}+O\big(\log^{2}\lambda\big)
\geq \sqrt{\lambda}\varphi_{2}(\lambda)Z_{\lambda}+\lambda\varphi_{3}(\lambda)\Big\}
\cap A_{\lambda}~\bigg\}.
\end{align}
\normalsize
We investigate \eqref{tazi3} for the following three sets:
\begin{align*}
&I_{f} = \{(x,y)\in[0,1]\times[0,1]\colon\,f(x)-f(y)> 0\},\\
&A_{f} = \{(x,y)\in[0,1]\times[0,1]\colon\,f(x)-f(y)< 0\},\\
&B_{f} = \{(x,y)\in[0,1]\times[0,1]\colon\,f(x)-f(y)= 0\}.
\end{align*}
Let $Z_{\lambda}^{\prime\prime}:=Z_{\lambda}^{\prime}\left( \varphi_{1}+O\big(\tfrac{\log\lambda}{\sqrt{\lambda}}\big)\right)^{1/2}$. Using the continuous mapping theorem and the central limit theorem, we have
\begin{align}\label{eq:clt:cmt}
\frac{1}{\lambda}Z_{\lambda}^{\prime\prime}
+O\big(\tfrac{\log\lambda}{\sqrt{\lambda}}\big)~\overset{\P_{\bar{\mu}, x, y}}\longrightarrow 0,
\qquad \lambda\to\infty.
\end{align}
Therefore, for $(x,y)\in I_{f}$,
\begin{align}
\label{tazi4}
&\P_{\bar{\mu}, x, y}\bigg\{\Big\{Z_{\lambda}^{\prime\prime}
+O\big(\log^{2}\lambda\big)\geq \sqrt{\lambda}\varphi_{2}(\lambda)Z_{\lambda}+\lambda\varphi_{3}(\lambda)\Big\}
\cap A_{\lambda}~\bigg\}\nonumber\\
&\leq\P_{\bar{\mu}, x, y}\bigg\{Z_{\lambda}^{\prime\prime}
+O\big(\log^{2}\lambda\big)\geq O\big(\sqrt{\lambda}\log\lambda\big)
+\lambda\varphi_{3}(\lambda)~\bigg\}\nonumber\\
&\leq \P_{\bar{\mu}, x,y}\bigg\{ \frac{ 1}{\lambda}Z_{\lambda}^{\prime\prime}+O\big(\tfrac{\log\lambda}{\sqrt{\lambda}}\big)
\geq \beta_{1}^{2}f(x)\big(f(x)-f(y)\big)~\bigg\}\nonumber\\
&\to 0, \qquad \lambda\to\infty.
\end{align}
For  $(x,y)\in A_{f}$, by \eqref{tazi2},
\begin{align}
\label{tazi5}
&\P_{\bar{\mu}, x,y }\bigg\{\Big\{Z_{\lambda}^{\prime\prime}
+O\big(\log^{2}\lambda\big)\geq \sqrt{\lambda}\varphi_{2}(\lambda)Z_{\lambda}+\lambda\varphi_{3}(\lambda)\Big\}
\cap A_{\lambda}~\bigg\}\nonumber\\
&=\P_{\bar{\mu}, x,y}\bigg\{\Big\{Z_{\lambda}^{\prime\prime}
+O\big(\log^{2}\lambda\big)\geq O\big(\sqrt{\lambda}\log\lambda\big)+\lambda\varphi_{3}(\lambda) \Big\}
\cap A_{\lambda}~\bigg\}\nonumber\\
&=\P_{\bar{\mu}, x,y}\bigg\{\Big\{\frac{ 1}{\lambda}Z_{\lambda}^{\prime\prime}+O\big(\tfrac{\log\lambda}{\sqrt{\lambda}}\big)
\geq \beta_{1}^{2}f(x)\big(f(x)-f(y)\big)\Big\}\cap A_{\lambda}~\bigg\}\nonumber\\
&\geq \P_{\bar{\mu}, x,y }\bigg\{ \frac{ 1}{\lambda}Z_{\lambda}^{\prime\prime}+O\big(\tfrac{\log\lambda}{\sqrt{\lambda}}\big)
\geq \beta_{1}^{2}f(x)\big(f(x)-f(y)\big)~\bigg\}\nonumber\\
&\quad -\P_{\bar{\mu}}\big\{A_{\lambda}^{c}\mid Q=x\big\}\to 1, \qquad \lambda\to\infty.
\end{align}
Next, let $(x,y)\in B_{f}$. Similarly as in the proof of Theorem \ref{thm1HER}, we consider the events 
\begin{align*}
B_{1,\lambda} = \Big\{Z_{\lambda}< - \frac{1}{\log\lambda}\Big\}, \quad
B_{2,\lambda} = \Big\{Z_{\lambda}> \frac{1}{\log\lambda}\Big\}, \quad
B_{3,\lambda} = \Big\{|Z_{\lambda}|\leq \frac{1}{\log\lambda}\Big\}.
\end{align*}
and note that there exists a standard Gaussian random variable $N$ defined on a probability space $(\Omega ,\mathcal{F},\P)$ such that, by \eqref{tazi2},
\begin{align}
&\P_{\bar{\mu}}\big\{A_{\lambda} \cap B_{1,\lambda} \mid Q=x\big\}\nonumber\\
&=\P_{\bar{\mu}}\big\{B_{1,\lambda}\mid Q=x \big\}
-\P_{\bar{\mu}}\big\{A^{c}_{\lambda} \cap B_{1,\lambda} \mid Q=x\big\} \to \P\big\{N<0 \big\}=\tfrac{1}{2},
\qquad \lambda\to\infty,
\label{tazi8}
\end{align}
and
\begin{align}
&\P_{\bar{\mu}}\big\{B_{2,\lambda}\mid Q=x\big\}\to  \P\big\{N>0 \big\}=\tfrac{1}{2},\nonumber\\
&\P_{\bar{\mu}}\big\{B_{3,\lambda}\mid Q=x\big\}\to  \P\big\{|N|\leq 0 \big\}=0.
\label{tazi9}
\end{align}
Similar to \eqref{eq:clt:cmt} we have 
\begin{align}
\label{eq:clt:cmt*}
Z_{\lambda}^{\prime\prime}\frac{\log\lambda}{\sqrt{\lambda}}
+O\big(\tfrac{\log^3\lambda}{\sqrt{\lambda}}\big)~\overset{\P_{\bar{\mu}, x, y}}\longrightarrow 0,
\qquad \lambda\to\infty.
\end{align}
Since $\tfrac{\log\lambda}{\sqrt{\lambda}}<\tfrac12 M_{-}$, we have
\begin{align}
\label{tazi10}
&\P_{\bar{\mu},x,y}\bigg\{\Big\{Z_{\lambda}^{\prime\prime}
+O\big(\log^{2}\lambda\big)\geq \sqrt{\lambda}\varphi_{2}(\lambda)Z_{\lambda}+\lambda\varphi_{3}(\lambda)\Big\}
\cap B_{2,\lambda}\,\bigg\}\nonumber\\
&\leq \P_{\bar{\mu}, x,y }\bigg\{Z_{\lambda}^{\prime\prime}\frac{\log\lambda}{\sqrt{\lambda}}
+O\big(\tfrac{\log^{3}\lambda}{\sqrt{\lambda}}\big)
> \beta_{1}f(x)\left(\sqrt{\beta_{1}}-\frac{\log\lambda}{\sqrt{\lambda}}\right)~\bigg\}\nonumber\\
&\leq \P_{\bar{\mu}, x, y}\bigg\{Z_{\lambda}^{\prime\prime}\frac{\log\lambda}{\sqrt{\lambda}}
+O\big(\tfrac{\log^{3}\lambda}{\sqrt{\lambda}}\big)
> \tfrac12 M_{-}^{3}\,\bigg\} \to 0, \qquad \lambda\to\infty.
\end{align}
Moreover, 
\begin{align*}
&\P_{\bar{\mu}, x,y}\bigg\{\Big\{Z_{\lambda}^{\prime\prime}
+O\big(\log^{2}\lambda\big)< \sqrt{\lambda}\varphi_{2}(\lambda)Z_{\lambda}+\lambda\varphi_{3}(\lambda)\Big\}
\cap B_{1,\lambda}~\bigg\}\nonumber\\
&\leq \P_{\bar{\mu}, x, y}\bigg\{Z_{\lambda}^{\prime\prime}\frac{\log\lambda}{\sqrt{\lambda}}
+O\big(\tfrac{\log^{3}\lambda}{\sqrt{\lambda}}\big)
<- \beta_{1}^{3/2}f^{3/2}(x)~\bigg\} \to 0, \qquad \lambda\to\infty.
\end{align*}
Therefore by \eqref{tazi8},
\begin{align}
\label{tazi11}
&\P_{\bar{\mu}, x, y}\bigg\{ \Big\{Z_{\lambda}^{\prime\prime}
+O\big(\log^{2}\lambda\big)\geq \sqrt{\lambda}\varphi_{2}(\lambda)Z_{\lambda}+\lambda\varphi_{3}(\lambda)\Big\}
\cap A_{\lambda} \cap B_{1,\lambda}~\bigg\}\nonumber\\
&=\P_{\bar{\mu}}\big\{A_{\lambda} \cap B_{1,\lambda}\mid Q=x \big\}\nonumber\\
&-\P_{\bar{\mu}, x, y}\bigg\{ \Big\{Z_{\lambda}^{\prime\prime}
+O\big(\log^{2}\lambda\big)< \sqrt{\lambda}\varphi_{2}(\lambda)Z_{\lambda}+\lambda\varphi_{3}(\lambda)\Big\}
\cap A_{\lambda} \cap B_{1,\lambda} ~\bigg\} \to \tfrac{1}{2}, \qquad \lambda\to\infty.
\end{align}
It follows from \eqref{tazi9}--\eqref{tazi11} that, for $(x,y)\in B_{f}$,
\begin{align}
\label{tazi12}
&\P_{\bar{\mu}, x,y}\bigg\{ \Big\{Z_{\lambda}^{\prime\prime}
+O\big(\log^{2}\lambda\big)\geq \sqrt{\lambda}\varphi_{2}(\lambda)Z_{\lambda}+\lambda\varphi_{3}(\lambda)\Big\}
\cap A_{\lambda} ~\bigg\} \to \tfrac{1}{2}, \qquad \lambda\to\infty.
\end{align}
Finally, from \eqref{tazi3}--\eqref{tazi5} and \eqref{tazi12} we get
\begin{align*}
&\lim_{\lambda\to\infty}\P_{\bar{\mu}, x, y}\bigg\{ \sum_{j=1}^{d_{\phi}}d_{j}
\geq d_{\phi}(d_{\phi}-1),d_{\phi}\geq 1~ \bigg\}
=\left\{
\begin{array}{ll}
0, &\text{if }(x,y)\in I_{f}, \\
1, &\text{if }(x,y)\in A_{f}, \\
\frac{1}{2}, &\text{if }(x,y)\in B_{f},
\end{array} 
\right.
\end{align*}
and so, by \eqref{tazi0}--\eqref{tazi00} and the dominated convergence theorem, we arrive at
\begin{align}\label{tazi13}
\lim_{\lambda\to\infty}\mu([0,\infty))=\P_{\bar{\mu}}\big\{(Q,Q^{\prime})\in A_{f}\big\}
+\tfrac{1}{2}\P_{\bar{\mu}}\big\{(Q,Q^{\prime})\in B_{f}\big\}.
\end{align}

\medskip\noindent
(c) By Fubini's theorem, we have
\begin{align*}
&\P_{\bar{\mu}}\big\{(Q,Q^{\prime})\in A_{f}\big\}
=\iint_{A_{f}}\dd x\,\dd y\,\beta_{1}^{-1}f(y) \geq \iint_{A_{f}}\dd x\,\dd y\,\beta_{1}^{-1}f(x)\\
&=1-\iint_{I_{f}}\dd x\,\dd y\,\beta_{1}^{-1}f(x)-\iint_{B_{f}}\dd x\,\dd y\,\beta_{1}^{-1}f(x)\\
&=1-\iint_{\substack{\{(y,x)\in[0,1]\times[0,1]:\\f(x)-f(y)> 0\}}}\dd y\,\dd x\,\beta_{1}^{-1}f(x)
-\iint_{B_{f}}\dd x\,\dd y\,\beta_{1}^{-1}f(x)\\
&=1-\iint_{A_{f}}\dd x\,\dd y\,\beta_{1}^{-1}f(y)-\iint_{B_{f}}\dd x\,\dd y\,\beta_{1}^{-1}f(x)\\
&=1-\P_{\bar{\mu}}\big\{(Q,Q^{\prime})\in A_{f}\big\}-\P_{\bar{\mu}}\big\{(Q,Q^{\prime})\in B_{f}\big\}.
\end{align*}
Hence
\begin{align}
\label{tazi14}
\P_{\bar{\mu}}\big\{(Q,Q^{\prime})\in A_{f}\big\}\geq
\tfrac{1}{2}-\tfrac{1}{2}\P_{\bar{\mu}}\big\{(Q,Q^{\prime})\in B_{f}\big\},
\end{align}
which together with \eqref{tazi13} implies that
\begin{align*}
\lim_{\lambda\to\infty}\mu([0,\infty))\geq\tfrac{1}{2}.
\end{align*}

\medskip\noindent
(d) If $\ell_{2}(A_{f})>0$, then $\geq$ in \eqref{tazi14} can be replaced by $>$. Following the same approach as in the proof of (c), we get $\lim_{\lambda\to\infty}\mu([0,\infty))>\frac{1}{2}$. But, if $\ell_{2}(A_{f})=0$, then it follows from \eqref{tazi13} that 
\begin{align*} 
\lim_{\lambda\to\infty}\mu([0,\infty))=\tfrac{1}{2}\P_{\bar{\mu}}\big\{(Q,Q^{\prime})\in B_{f}\big\}.
\end{align*}
This, together with part (c), leads to the conclusion that $\lim_{\lambda\to\infty}\mu([0,\infty))=\frac{1}{2}$.

\medskip\noindent
(e) If $f$ is injective, then $B_{f}=\{(x,y)\in [0,1]\times [0,1]\colon\,x=y\}$ and so
\begin{align*}
\P_{\bar{\mu}}\big\{(Q,Q^{\prime})\in B_{f}\big\}=\iint_{B_{f}}\dd x\dd y~\beta_{1}^{-1}f(y)=0.
\end{align*}
Therefore, by \eqref{tazi13},
\begin{align*}
\lim_{\lambda\to\infty}\mu([0,\infty))=\P_{\bar{\mu}}\big\{(Q,Q^{\prime})\in A_{f}\big\}.
\end{align*}


\paragraph{Proof of Theorem \ref{thm2IER}.}

Let $\nu_{0}(k)$ denote the probability that the root has $k$ children. Since
\begin{align*}
&\E_{\bar{\mu}}\big[d_{\phi}\big] =\int_{[0,1]}\dd x\,\E_{\bar{\mu}}\big[d_{\phi}\mid Q=x\big]
=\int_{[0,1]}\dd x\, \lambda\beta_{1}f(x)=\lambda\beta_{1}^{2},\\
&\E_{\bar{\mu}}\big[d_{1}\big] =\int_{[0,1]}\dd x\,\beta_{1}^{-1}f(x)\E_{\bar{\mu}}\big[d_{1}\mid Q^{\prime}=x\big]
=\int_{[0,1]}\dd x\, \lambda f^{2}(x)=\lambda\beta_{2},
\end{align*}
similarly as in the proof of Theorem \ref{thm2HER} we obtain
\begin{align*}
\E_{\bar{\mu}}\big[\Delta_{\phi}\big] 
=\sum_{k\in\N}\nu_{0}(k)\,\E_{\bar{\mu}}\bigg[\frac{1}{k}\sum_{j=1}^{k}d_{j}+1-k\bigg]
&=\E_{\bar{\mu}}\big[d_{1}\big]\big(1-\nu_{0}(0)\big)+1-\nu_{0}(0)-\E_{\bar{\mu}}\big[d_{\phi}\big]\\
&= 1-(\lambda\beta_{2}+1)\int_{[0,1]} \dd x\,\ee^{-\lambda \beta_{1} f(x)}
+ \lambda(\beta_{2}-\beta_{1}^{2}).
\end{align*}
Since 
\begin{align*}
1-(\lambda M_{+}^{2}+1)\,\ee^{-\lambda M_{-}^{2}} + \lambda(M_{-}^{2}-M_{+}^{2})
\leq \E_{\bar{\mu}}\big[\Delta_{\phi}\big] 
\leq 1-(\lambda M_{-}^{2}+1)\,\ee^{-\lambda M_{+}^{2}}
+ \lambda(M_{+}^{2}-M_{-}^{2}),
\end{align*}
it follows that $\E_{\bar{\mu}}\big[\Delta_{\phi}\big] \to 0$ as $\lambda\downarrow 0$. On the other hand, if $U$ is a uniform $[0,1]$ random variable, then
\begin{align*}
\beta_{2}-\beta_{1}^{2}=\mathbb{V}\mathrm{ar}(f(U))>0
\end{align*}
by noting that $f$ is a non-constant function. Since
\begin{align*}
\E_{\bar{\mu}}\big[\Delta_{\phi}\big] \geq 1-(\lambda M_{+}^{2}+1)\ee^{-\lambda M_{-}^{2}}
+ \lambda (\beta_{2}-\beta_{1}^{2}),
\end{align*}
it follows that $\E_{\bar{\mu}}[\Delta_{\phi}] \to \infty$ as $\lambda\to \infty$. Similarly, since 
\begin{align*}
&\E_{\bar{\mu}}\big[d_{\phi}^{2}\big] =\int_{[0,1]} \dd x\,\E_{\bar{\mu}}\big[d_{\phi}^{2}\mid Q=x\big]
=\int_{[0,1]} \dd x\, \big(\lambda\beta_{1}f(x)+\lambda^{2}\beta_{1}^{2}f^{2}(x)\big)
=\lambda\beta_{1}^{2}+\lambda^{2}\beta_{1}^{2}\beta_{2},\\
&\E_{\bar{\mu}}\big[d_{1}^{2}\big] =\int_{[0,1]} \dd x\,\beta_{1}^{-1}f(x)\,\E_{\bar{\mu}}
\big[d_{1}^{2}\mid Q^{\prime}=x\big]=\int_{[0,1]} \dd x\, \big(\lambda f^{2}(x)
+\lambda^{2}\beta_{1}f^{3}(x)\big)=\lambda\beta_{2}+\lambda^{2}\beta_{1}\beta_{3},
\end{align*}
it follows that
\begin{align*}
\E_{\bar{\mu}}\big[\Delta_{\phi}^{2}\big] 
&= \sum_{k\in\N} \nu_{0}(k)\,\E_{\bar{\mu}}\big[\Delta_{\phi}^{2}\mid d_{\phi}=k\big]
=\sum_{k\in\N} \nu_{0}(k)\,\E_{\bar{\mu}}\bigg[\Big(\frac{1}{k}\sum_{j=1}^{k}d_{j}+1-k\Big)^{2}\bigg]\\
&=\sum_{k\in\N} \nu_{0}(k)\,\E_{\bar{\mu}}\bigg[\frac{1}{k^{2}}\sum_{j=1}^{k}d_{j}^{2}
+\frac{1}{k^{2}}\sum_{i=1}^{k}\sum_{j=1 \atop {j\neq i}}^{k}d_{i}d_{j}+1+k^{2}
+\dfrac{2}{k}\sum_{j=1}^{k}d_{j}-2\sum_{j=1}^{k}d_{j}-2k\bigg]\\
&=\mathbb{V}\mathrm{ar}_{\bar{\mu}}\big(d_{1}\big)\,\E_{\bar{\mu}}\Big[d_{\phi}^{-1}
\mathbbm{1}_{\{d_{\phi}\neq 0\}}\Big]+\big(\E_{\bar{\mu}}\big[d_{1}\big]\big)^{2}\big(1-\nu_{0}(0)\big)+1-\nu_{0}(0)\\
&\quad\quad+\E_{\bar{\mu}}\big[d_{\phi}^{2}\big]+2\E_{\bar{\mu}}\big[d_{1}\big]\big(1-\nu_{0}(0)\big)
-2\E_{\bar{\mu}}\big[d_{1}\big]\,\E_{\bar{\mu}}\big[d_{\phi}\big]-2\E_{\bar{\mu}}\big[d_{\phi}\big]\\
&=(\lambda\beta_{2}+\lambda^{2}\beta_{1}\beta_{3}-\lambda^{2}\beta_{2}^{2})\,
\E_{\bar{\mu}}\big[d_{\phi}^{-1}\mathbbm{1}_{\{d_{\phi}\neq 0\}}\big]
-(1+\lambda^{2}\beta_{2}^{2}+2\lambda\beta_{2})\int_{0}^{1} \dd x\,\mathrm{e}^{-\lambda \beta_{1} f(x)}\\
&\quad\quad+\lambda^{2}(\beta_{2}^{2}-\beta_{1}^{2}\beta_{2})+\lambda(2\beta_{2}-\beta_{1}^{2})+1\\
&=\lambda\big(\beta_{2}+\lambda(\beta_{1}\beta_{3} -\beta_{2}^{2})\big)\int_{[0,1]} \dd x\,
\sum_{k\in\N}k^{-1}\dfrac{\ee^{-\lambda \beta_{1}f(x)}(\lambda \beta_{1}f(x))^{k}}{k!}\\
&\quad\quad-(1+\lambda^{2}\beta_{2}^{2}+2\lambda\beta_{2})\int_{[0,1]} \dd x\,
\mathrm{e}^{-\lambda \beta_{1} f(x)} +\lambda^{2}(\beta_{2}^{2}-\beta_{1}^{2}\beta_{2})
+\lambda(2\beta_{2}-\beta_{1}^{2})+1.
\end{align*}
Again, since
\begin{align*}
&-(1+\lambda^{2}M_{+}^{4}+2\lambda M_{+}^{2})\,\mathrm{e}^{-\lambda M_{-}^{2}}+\lambda^{2}(M_{-}^{4}-M_{+}^{4})
+\lambda(2M_{-}^{2}-M_{+}^{2})+1\\
&\qquad\qquad\qquad\qquad\qquad \leq\E_{\bar{\mu}}\big[\Delta_{\phi}^{2}\big] 
\leq (\lambda M_{+}^{2}+\lambda^{2}M_{+}^{4}) - \mathrm{e}^{-\lambda M_{+}^{2}}
+\lambda^{2}M_{+}^{4}+2\lambda M_{+}^{2}+1,
\end{align*}
it follows that $\E_{\bar{\mu}}\big[\Delta^{2}_{\phi}\big] \to 0$ as $\lambda\downarrow 0$. In addition, $\lim_{\lambda\to\infty} \E_{\bar{\mu}}[\Delta_{\phi}^{2}] \geq \lim_{\lambda\to\infty}(\E_{\bar{\mu}}[\Delta_{\phi}])^{2}=\infty$.
 

\paragraph{Proof of Theorem \ref{thm3IER}.}

Fix $\lambda \in (0,\infty)$. First, we show that for each $k\in\N$,
\begin{align}\label{tazi16}
\P_{\bar{\mu}}\bigg\{\sum_{j=1}^{k}d_{j}\geq kx+k(k-1)\bigg\} 
\sim \P_{\bar{\mu}}\bigg\{\sum_{j=1}^{k}d_{j}= kx+k(k-1)\bigg\}, \qquad x \to \infty.
\end{align}
Indeed, without loss of generality we may assume that $x\in\N$. With the same notations and methods used in the proof of Theorem \ref{thm1IER}, we have, for $k\in\N$ and $m\in\N_{0}$,
\begin{align*}
&\P_{\bar{\mu}}\bigg\{\sum_{j=1}^{k}d_{j}= kx+k(k-1)+m\bigg\}\\
&= \beta_{1}^{-k}\prod_{j=1}^{k} \int_{[0,1]} \dd y_{j}\,f(y_{j})\, 
\dfrac{\ee^{-\lambda \beta_{1}
\sum_{j=1}^{k} f(y_{j})}(\lambda\beta_{1}\sum_{j=1}^{k} f(y_{j}))^{kx+k(k-1)+m}}{(kx+k(k-1)+m)!}.
\end{align*}
Therefore, for $k,l\in\N$,
\begin{align*}
&\dfrac{\P_{\bar{\mu}}\big\{\sum_{j=1}^{k}d_{j}= kx+k(k-1)+l\big\}}
{\P_{\bar{\mu}}\big\{\sum_{j=1}^{k}d_{j}= kx+k(k-1)\big\}}\\
&\leq\dfrac{(\lambda  \beta_{1})^{l}\ee^{-\lambda \beta_{1}k(M_{-}-M_{+})}}{(kx+k(k-1)+1)\times\cdots\times(kx+k(k-1)+l)}\\
&\qquad \times\dfrac{\prod_{j=1}^{k} \int_{[0,1]} \dd y_{j}\,f(y_{j}) (\sum_{j=1}^{k} f(y_{j}))^{kx+k(k-1)+l}}
{\prod_{j=1}^{k} \int_{[0,1]} \dd y_{j}\,f(y_{j}) (\sum_{j=1}^{k} f(y_{j}))^{kx+k(k-1)}}\\
&\leq\dfrac{(\lambda  \beta_{1}k M_{+})^{l}\ee^{-\lambda \beta_{1}k(M_{-}-M_{+})}}{(kx+k(k-1)+1)
\times\cdots\times(kx+k(k-1)+l)}
\to 0, \qquad x\to\infty,
\end{align*}
and 
\begin{align*}
&\sum_{l\in\N}\dfrac{\P_{\bar{\mu}}\big\{\sum_{j=1}^{k}d_{j}= kx+k(k-1)+l\big\}}
{\P_{\bar{\mu}}\big\{\sum_{j=1}^{k}d_{j}= kx+k(k-1)\big\}}\\
&\leq \ee^{-\lambda \beta_{1}k(M_{-}-M_{+})}\sum_{l\in\N}\dfrac{(\lambda \beta_{1}k M_{+})^{l}}{l!}
\leq\ee^{-\lambda \beta_{1}k(M_{-}-2M_{+})}<\infty.
\end{align*}
Therefore, \eqref{tazi16} follows by the dominated convergence theorem.

For $k\in\N\setminus\{1\}$, since
\begin{align*}
&\prod_{j=1}^{k} \int_{[0,1]} \dd y_{j}\,f(y_{j}) \Big(\sum_{j=1}^{k} f(y_{j})\Big)^{kx}
\leq k^{kx}\prod_{j=1}^{k} \int_{[0,1]} \dd y_{j}\,f(y_{j}) \sum_{j=1}^{k} f^{kx}(y_{j})\\
&\leq k^{kx}M_{+}^{kx-x+k-1}\prod_{j=1}^{k} \int_{[0,1]} \dd y_{j}\,\sum_{j=1}^{k} f^{x+1}(y_{j})
=k^{kx+1}M_{+}^{kx-x+k-1}\beta_{x+1},
\end{align*}
by Stirling's formula we have
\begin{align}
\label{tazi20}
&\dfrac{\P_{\bar{\mu}}\big\{\sum_{j=1}^{k}d_{j}= kx+k(k-1)\big\}}{\P_{\bar{\mu}}\big\{d_1=x \big\}}\nonumber\\
&\leq\dfrac{x!}{(kx)!}\times\dfrac{\ee^{-\lambda \beta_{1}(kM_{-}-M_{+})}\beta_{1}^{-k+1}
(\lambda\beta_{1}kM_{+})^{k(k-1)}(\lambda\beta_{1})^{kx-x}}{(kx+1)\times\cdots\times(kx+k(k-1))}\nonumber\\
&\qquad \times\dfrac{\prod_{j=1}^{k} \int_{[0,1]} \dd y_{j}\,f(y_{j}) 
(\sum_{j=1}^{k} f(y_{j}))^{kx}}{\int_{[0,1]} \dd y_{1}\, (f(y_{1}))^{x+1}}\nonumber\\
&\leq\dfrac{x!}{(kx)!}\times\dfrac{\ee^{-\lambda \beta_{1}(kM_{-}-M_{+})}\beta_{1}^{-k+1}
(\lambda\beta_{1}kM_{+})^{k(k-1)}(\lambda\beta_{1})^{kx-x}}{(kx+1)\times\cdots\times(kx+k(k-1))}
\times k^{kx+1}M_{+}^{kx-x+k-1}\nonumber\\
&\sim\left(\frac{x}{\ee M_{+}}\right)^{-x(k-1)}
\times\dfrac{\ee^{-\lambda \beta_{1}(kM_{-}-M_{+})}\beta_{1}^{-k+1}(\lambda\beta_{1}kM_{+})^{k(k-1)}
(\lambda\beta_{1})^{kx-x}}{(kx+1)\times\cdots\times(kx+k(k-1))}\times \sqrt{k} M_{+}^{k-1}\nonumber\\
&\to 0, \qquad x\to\infty .
\end{align}
Therefore, taking the same approach as in \eqref{forIHEproof}, we get from \eqref{tazi16}--\eqref{tazi20} and the dominated convergence theorem that
\begin{align*}
\P_{\bar{\mu}}\{\Delta_{\phi}\geq x\}\sim \P_{\bar{\mu}}\big\{d_{\phi}=1 \big\} \P_{\bar{\mu}}\big\{d_{1}\geq x\big\},
\qquad x \to \infty.
\end{align*}
Hence, from \eqref{tazi16},
\begin{align*}
\P_{\bar{\mu}}\big\{\Delta_{\phi}\geq x\big\}
&\sim\P_{\bar{\mu}}\big\{d_{\phi}=1 \big\} \P_{\bar{\mu}}\big\{d_{1}= x\big\}, \qquad x \to \infty,
\end{align*}
and so Stirling's formula yields
\begin{align*}
&\left(\frac{ \lambda M_{-}^{2}\ee^{-2\lambda\beta_{1}M_{+}}}{\sqrt{2\pi}}\right)
\frac{ 1}{\sqrt{ x}}\beta_{x}\ee^{-x\log(\frac{x}{\lambda\beta_{1}\ee})}\lesssim\P_{\bar{\mu}}\big\{\Delta_{\phi}\geq x\big\}
\lesssim\left(\frac{ \lambda M_{+}^{2}\ee^{-2\lambda\beta_{1}M_{-}}}{\sqrt{2\pi}}\right)
\frac{ 1}{\sqrt{ x}}\beta_{x}\ee^{-x\log(\frac{x}{\lambda\beta_{1}\ee})},\\ 
&\qquad x \to\infty,
\end{align*}
where we recall that $\beta_x = \int_{[0,1]} \dd y\,f^x(y)$.

\medskip\noindent
(a) Suppose that $\ell_{1}(\{y\in [0,1]\colon\,f(y)=M_{+}\})>0$ with $\ell_{1}$ the Lebesgue measure. Then 
\begin{align*}
\ell_{1}\big(\{y\in [0,1]\colon\,f(y)=M_{+}\}\big)M_{+}^{x} \leq \beta_{x}\leq M_{+}^{x},
\end{align*}
and hence $\beta_{x}\asymp M_{+}^{x}$ as $x\to\infty$. 

\medskip\noindent
(b) Suppose that $f(y_{\star})=M_{+}$ and $f(y_{\star})-f(y) \asymp |y_{\star}-y|^\alpha$, $y \to y_{\star}$, for some $y_{\star} \in [0,1]$ and $\alpha \in (0,\infty)$, and $f$ is bounded away from $M_{+}$ outside any neighbourhood of $y_{\star}$. Then
\begin{equation*}
\begin{aligned}
\frac{\beta_{x}}{M_{+}^{x}} &= \int_{[0,1]} \dd y\,\left(\frac{f(y)}{f(y_{\star})}\right)^x 
= \int_{[0,1]} \dd y\,\left(1-\frac{f(y_{\star})-f(y)}{f(y_{\star})}\right)^x\\
&\asymp \int_{[0,1]} \dd y\,\exp\left(-x\,\frac{|y_{\star}-y|^\alpha}{M_{+}}\right)
\asymp \int_{y_{\star}-x^{-1/\alpha}}^{y_{\star}+x^{-1/\alpha}} \dd y \asymp x^{-1/\alpha},
\qquad x \to \infty,
\end{aligned}
\end{equation*}
where the first $\asymp$ uses that only the neighbourhood of $y_{\star}$ contributes, and the second $\asymp$ uses that the integral is dominated by the neighbourhood where the exponent is of order $1$. Hence $\beta_{x} \asymp x^{-1/\alpha}\,M_{+}^{x}$ as $x\to\infty$.


\subsection{Configuration model}


\paragraph{Proof of Theorem~\ref{thm1CM}.}

Write 
\begin{align*}
\mu([0,\infty)) = \sum_{k\in\N} p_{k} \nu_{1}^{\circledast k}\big([k(k-1),\infty)\big),
\end{align*}
where $\nu_{1}^{\circledast k}$ is the $k$-fold convolution of $\nu_1=p^{\star}$. We look at the special case where $p_{k} = k^{-\tau}/\zeta(\tau)$ with $\tau \in (2,\infty)$ and $\zeta$ the Riemann zeta function. Then,
\begin{align*}
\mu([0,\infty)) \geq p_{1}=\dfrac{1}{\zeta(\tau)}\geq\dfrac{1}{\zeta(2)}=\dfrac{6}{\pi^{2}}>\dfrac{1}{2}.
\end{align*}

For $m\in\N$ and $s \in (1,\infty)$, let $\zeta_{m}(s) = \sum_{j=m}^{\infty} j^{-s}$ be the truncated Riemann zeta function. Then\footnote{It is worthwhile to note here that $\zeta_{1}(s)> 1$, and for $m> 1$, there is a sharper upper bound which will be presented in \eqref{truncatedRiemann2}.}
\begin{align}\label{truncatedRiemann}
\dfrac{m^{-s+1}}{s-1} = \int_{m}^{\infty} \dd x\,x^{-s} \leq \zeta_{m}(s)\leq 1+\int_{m}^{\infty} \dd x\,x^{-s}
= \dfrac{s-1+m^{-s+1}}{s-1}.
\end{align}
For each $k\in\N$,
\begin{align*}
\nu_{1}^{\circledast k}\big([k(k-1),\infty)\big)\geq \P_{\bar{\mu}}\big\{d_{1}\geq k(k-1)\big\}
=\dfrac{\zeta_{k(k-1)+1}(\tau -1)}{\zeta_{1}(\tau -1)}\geq \dfrac{\zeta_{k^{2}}(\tau -1)}{\zeta_{1}(\tau -1)}
\geq \dfrac{k^{-2\tau + 4}}{\tau -1}.
\end{align*}
Consequently, 
\begin{align*}
\mu([0,\infty)) \geq \dfrac{1}{\tau -1} \sum_{k\in\N} p_{k}\,k^{-2\tau + 4} \geq 
\dfrac{1}{\tau -1}\dfrac{\zeta_{1}(3\tau - 4)}{\zeta_{1}(\tau)}.
\end{align*}
By the dominated convergence theorem, $\zeta_{1}(3\tau - 4) \to \zeta_{1}(2)$ as $ \tau\downarrow 2$, from which we obtain that $\mu([0,\infty)) \to 1$ as $\tau\downarrow 2$, and
\begin{align*}
\mu([0,\infty)) \geq \dfrac{1}{\zeta(\tau)}\to 1, \qquad \tau\to\infty,
\end{align*}
from which we obtain that $\mu([0,\infty)) \to 1$ as $\tau\to \infty$. 


\paragraph{Proof of Theorem~\ref{thm2CM}.}

Let $\E[D]<\infty$. Since 
\begin{align*}
\E_{\bar{\mu}}[d_{\phi}] = \E[D], \qquad \E_{\bar{\mu}}[d_{1}] = (\E[D])^{-1}(\E[D^{2}]-\E[D]), 
\end{align*}
the first moment equals
\begin{align*}
\E_{\bar{\mu}}\big[\Delta_{\phi}\big] = \sum_{k\in\N} p_{k}\left(\dfrac{1}{k}\sum_{j=1}^{k}\E_{\bar{\mu}}[d_{j}]+1-k\right)
=\E_{\bar{\mu}}[d_{1}]+1-\E_{\bar{\mu}}[d_{\phi}]=\dfrac{\mathbb{V}\mathrm{ar}(D)}{\E[D]},
\end{align*}
which is strictly positive whenever $D$ is a non-degenerate random variable, i.e., the limit is not a regular tree. Moreover, $\E_{\bar{\mu}}[d_{\phi}^{-1}] = \E[D^{-1}]$, $\E_{\bar{\mu}}[d_{\phi}^{2}] = \E[D^{2}]$, and if also $\E[D^2]<\infty$, then
\begin{align*}
\E_{\bar{\mu}}\big[d_{1}^{2}\big]=\dfrac{\sum_{k\in\N}(k-1)^{2}kp_{k}}{\E[D]}=\dfrac{\E[D^{3}]+\E[D]-2\E[D^{2}]}{\E[D]}.
\end{align*}
Therefore, if $\E[D^2]<\infty$, then
\begin{align*}
&\E_{\bar{\mu}}\big[\Delta_{\phi}^{2}\big]\\ 
&= \sum_{k\in\N}p_{k}\E_{\bar{\mu}}\big[\Delta_{\phi}^{2}\mid d_{\phi}=k\big]
=\sum_{k\in\N}p_{k}\E_{\bar{\mu}}\bigg[\Big(\frac{1}{k}\sum_{j=1}^{k}d_{j}+1-k\Big)^{2}\bigg]\\
&=\sum_{k\in\N}p_{k}\E_{\bar{\mu}}\bigg[\frac{1}{k^{2}}\sum_{j=1}^{k}d_{j}^{2}+\frac{1}{k^{2}}
\sum_{i=1}^{k}\sum_{j=1 \atop {j\neq i}}^{k}d_{i}d_{j}+1+k^{2}+\dfrac{2}{k}\sum_{j=1}^{k}d_{j}
-2\sum_{j=1}^{k}d_{j}-2k\bigg]\\
&=\mathbb{V}\mathrm{ar}_{\bar{\mu}}\big(d_{1}\big)\E_{\bar{\mu}}\big[d_{\phi}^{-1}\big]
+\big(\E_{\bar{\mu}}\big[d_{1}\big]\big)^{2}+1+\E_{\bar{\mu}}\big[d_{\phi}^{2}\big]
+2\E_{\bar{\mu}}\big[d_{1}\big]
-2\E_{\bar{\mu}}\big[d_{1}\big]\E_{\bar{\mu}}\big[d_{\phi}\big]-2\E_{\bar{\mu}}\big[d_{\phi}\big]\\
&= \dfrac{\left(\E[D^{3}]\E[D]-(\E[D^{2}])^{2}\right)
\E[D^{-1}]+\E[D^{2}]\mathbb{V}\mathrm{ar}(D)}{(\E[D])^{2}},
\end{align*}
and so using the Cauchy-Schwarz inequality, 
\begin{align*}
\E_{\bar{\mu}}\big[\Delta_{\phi}^{2}\big] 
\geq \dfrac{\E[D^{2}]\mathbb{V}\mathrm{ar}(D)}{(\E[D])^{2}},
\end{align*}
which implies that
\begin{align*}
\mathbb{V}\mathrm{ar}_{\bar{\mu}}(\Delta_{\phi}) \geq \mathbb{V}\mathrm{ar}(D)\dfrac{\E[D^{2}]
-\mathbb{V}\mathrm{ar}(D)}{(\E[D])^{2}}=\mathbb{V}\mathrm{ar}(D). 
\end{align*}


\paragraph{Proof of Theorem \ref{thm3CM}.}

Let $x\in (0,\infty)$ and $\tau \in (2,\infty)$. For $m\in\N$ and $s\in(1,\infty)$,
\begin{align}
\label{truncatedRiemann2}
\zeta_{m}(s)=\sum_{k=m}^{\infty} k^{-s}\leq \int_{m-1}^{\infty} \dd x\,x^{-s} = \dfrac{(m-1)^{-s+1}}{s-1}.
\end{align}
Therefore
\begin{align*}
\P_{\bar{\mu}}\big\{d_{1}\geq x\big\}=\dfrac{1}{\zeta(\tau)\,\E[D]}\sum_{k=\lceil x\rceil}^{\infty}(k+1)^{-\tau+1}\leq\dfrac{1}{\E[D]}
\sum_{k=\lceil x\rceil +1}^{\infty}k^{-\tau+1}\leq \dfrac{1}{\E[D]}\,\dfrac{x^{-\tau +2}}{\tau -2}
\end{align*}
and
\begin{align*}
&\P_{\bar{\mu}}\{\Delta_{\phi}\geq x\}=\sum_{k\in\N}\P_{\bar{\mu}}\big\{d_{\phi}=k \big\}
\P_{\bar{\mu}}\Big\{\sum_{j=1}^{k}d_{j}\geq k x +k(k-1)\Big\}\\
&\leq\sum_{k\in\N}\P_{\bar{\mu}}\big\{d_{\phi}=k \big\}\P_{\bar{\mu}}\Big\{\sum_{j=1}^{k}d_{j}\geq k x \Big\}
\leq\sum_{k\in\N}k\P_{\bar{\mu}}\big\{d_{\phi}=k \big\}\P_{\bar{\mu}}\big\{d_{1}\geq x \big\}
\leq (\tau -2)^{-1}x^{-\tau +2}.
\end{align*}
Moreover, by \eqref{truncatedRiemann}, for $y>0$,
\begin{align*}
\P_{\bar{\mu}}\big\{d_{1}\geq y\big\}&=\sum_{k=\lceil y\rceil}^{\infty}\dfrac{(k+1)^{-\tau +1}}{\zeta(\tau-1)}=\dfrac{\zeta_{\lceil y\rceil +1}(\tau-1)}{\zeta(\tau-1)}\geq \dfrac{(y+2)^{-\tau+2}}{(\tau -2)\zeta(\tau-1)}.
\end{align*}
Therefore
\begin{align*}
\P_{\bar{\mu}}\big\{d_{1}\geq x\big\}\asymp x^{-\tau+2},\quad\quad x\to\infty,
\end{align*}
and
\begin{align*}
&\P_{\bar{\mu}}\{\Delta_{\phi}\geq x\} \geq \sum_{k\in\N}\P_{\bar{\mu}}\big\{d_{\phi}=k \big\}
\P_{\bar{\mu}}\big\{d_{1}\geq k^{2}(x+1)\big\}\\
&\geq \dfrac{1}{(\tau -2)\zeta(\tau)\zeta(\tau-1)}\sum_{k\in\N}k^{-\tau}\big(k^{2}(x+1)+2\big)^{-\tau +2}\\
&\geq \dfrac{3^{-\tau+2}\zeta(3\tau-4)}{(\tau -2)\zeta(\tau)\zeta(\tau-1)}(x+1)^{-\tau+2}
\geq \dfrac{3^{-\tau+2}}{\tau (3\tau -5)}(x+1)^{-\tau+2}
\geq 3^{-\tau+1}\tau^{-2}(x+1)^{-\tau+2}.
\end{align*}


\subsection{Preferential attachment model}


\paragraph{Proof of Theorem~\ref{thm1PAM}.}

First note that
\begin{equation*}
d_{\phi}=\no(\phi)+\ny(\phi)=1+\ny(\phi).
\end{equation*}
Let $d_{i}$ denote the number of children of $\phi_{i} = \phi i$. We have $d_{1}=1+\ny(\phi_{1})$ and, when $\ny(\phi)\geq 1$,
\begin{align*}
d_{i}=\ny(\phi_{i}),\qquad  i\in\{2,\ldots ,d_{\phi}\}.
\end{align*}
For $a\in[0,1]$, define 
\begin{align*}
\kappa (a) = \int_{[a,1]} \dd x\,\dfrac{1}{2+\delta}\dfrac{x^{-(1+\delta)/(2+\delta)}}{a^{1/(2+\delta)}}
= a^{-\frac{1}{2+\delta}}-1.
\end{align*}
For $b \in (0,\infty)$, let $Y(a,b) \overset{d}= \mathrm{Poisson}(b\kappa (a))$. Then
\begin{align*}
\mu([0,\infty)) = \sum_{k\in\N}  \P_{\bar{\mu}}\{d_{\phi}=k\} 
\P_{\bar{\mu}}\left\{ \sum_{j=1}^{k}d_{j}\geq k(k-1)~\mid~d_{\phi}=k \right\}.
\end{align*}
Recall that $\Gamma_{\phi}\overset{d}= \text{Gamma}(1+\delta, 1)$, hence
\begin{align}
\label{a-04}
\P_{\bar{\mu}}\{d_{\phi}=k\mid A_{\phi}=a\}
&= \int_{0}^{\infty} \dd b\,\P_{\bar{\mu}}\big\{\ny(\phi)=k-1\mid A_{\phi}=a,\Gamma_{\phi}=b\big\}
\dfrac{b^{\delta} \ee^{-b}}{\Gamma(1+\delta)}\nonumber\\
&=  \int_{0}^{\infty} \dd b\,\P\big\{Y(a,b)=k-1\big\}\dfrac{b^{\delta} 
\ee^{-b}}{\Gamma(1+\delta)}\nonumber\\
&=\dfrac{\Gamma(k+\delta)}{(k-1)!\,\Gamma(1+\delta)}\Big(1-a^{\frac{1}{2+\delta}}\Big)^{k-1}a^{\frac{1+\delta}{2+\delta}},
\end{align}
and so
\begin{align}
\label{a-4}
\P_{\bar{\mu}}\{d_{\phi}=k\}
&= \dfrac{\Gamma(k+\delta)}{(k-1)!\,\Gamma(1+\delta)}\int_{[0,1]} \dd a \,\Big(1-a^{\frac{1}{2+\delta}}\Big)^{k-1}a^{\frac{1+\delta}{2+\delta}}\nonumber\\
&= \dfrac{(2+\delta)\Gamma(k +\delta)}{(k-1)!\,\Gamma(1+\delta)}\int_{[0,1]} \dd u \,(1-u)^{k-1}u^{2+2\delta}\nonumber\\
&= \dfrac{(2+\delta)\Gamma(3+2\delta)\Gamma(k+\delta)}{\Gamma(1+\delta)
\Gamma(k+3+2\delta)}.
\end{align}
Hence
\begin{align}
\label{a-5}
\mu([0,\infty))\geq \P_{\bar{\mu}}\{d_{\phi}=1\}=\dfrac{2+\delta}{3+2\delta},
\end{align}
and so $\mu([0,\infty))\to 1$ as $\delta \downarrow -1$. Moreover, since $ \tfrac{2+\delta}{3+2\delta}>\frac12$, it follows from \eqref{a-5} that, for all $\delta \in (-1,\infty)$,
\begin{align*}
\mu([0,\infty)) > \dfrac{1}{2}.
\end{align*}


\paragraph{Proof of Theorem~\ref{thm2PAM}.}

Let $\kappa(\cdot)$ and $Y(\cdot,\cdot)$ be as in the proof of Theorem \ref{thm1PAM}. Also, let $d_{i}$ denote the number of children of $\phi_{i} = \phi i$. Then $\ny(\phi)$ is distributed as $Y(A_{\phi},\Gamma_{\phi})$,
and hence
\begin{align}
\label{a-7}
\E_{\bar{\mu}}\big[d_{\phi}\big] = 1+\E_{\bar{\mu}}\big[\ny(\phi)\big] 
= 1+\E_{\bar{\mu}}[\Gamma_{\phi}]\E_{\bar{\mu}}[\kappa(A_{\phi})]=2.
\end{align}

Abbreviate
\begin{equation*}
\gamma^\dagger_{k,\delta} = \dfrac{\Gamma(k+3+2\delta)}{(k-1)!\Gamma(3+2\delta)}.
\end{equation*}
For $k\in\N$, using \eqref{a-04}--\eqref{a-4}, we get
\begin{align*}
\E_{\bar{\mu}}\big[A_{\phi}^{-\frac{1}{2+\delta}}\mid d_{\phi}=k\big]\
&= \int_{[0,1]} \dd a\, \dfrac{\P_{\bar{\mu}}\{d_{\phi}=k\mid A_{\phi}=a\}}{\P_{\bar{\mu}}\{d_{\phi}=k\}}a^{-\frac{1}{2+\delta}}\\
&= \dfrac{\gamma^\dagger_{k,\delta}}{2+\delta}
\int_{[0,1]} \dd a\,(1-a^{\frac{1}{2+\delta}})^{k-1}a^{\frac{\delta}{2+\delta}}\\
&= \gamma^\dagger_{k,\delta}
\int_{[0,1]} \dd u\,(1-u)^{k-1}u^{1+2\delta}
= \dfrac{k+2+2\delta}{2+2\delta}.
\end{align*}
Hence, considering $\phi_1$, we have
\begin{align*}
\E_{\bar{\mu}}\big[\kappa(A_{\phi_{1}})\mid d_{\phi}=k\big]
=\E_{\bar{\mu}}\big[U_{\phi_{1}}^{-\frac{1}{1+\delta}}\big]
\E_{\bar{\mu}}\big[A_{\phi}^{-\frac{1}{2+\delta}}\mid d_{\phi}=k\big]-1
=\left\{
\begin{array}{ll}
\frac{k+2}{2\delta}, &\text{if }\delta>0, \\
\infty, &\text{if }-1<\delta\leq 0.
\end{array} 
\right.
\end{align*}
Thus, observing that $\ny(\phi_{1})$ is distributed as $Y(A_{\phi_{1}},\Gamma_{\phi_{1}})$, we have
\begin{align}
\label{a-6}
\E_{\bar{\mu}}\big[d_{1}\mid d_{\phi}=k\big] &=1+\E_{\bar{\mu}}\big[\ny(\phi_{1})\mid d_{\phi}=k\big] \nonumber\\&=
1+\E_{\bar{\mu}}[\Gamma_{\phi_{1}}]\, \E_{\bar{\mu}}\big[\kappa(A_{\phi_{1}})
\mid d_{\phi}=k\big]
=\left\{
\begin{array}{ll}
1+\tfrac{2+\delta}{2\delta}(k+2), &\text{if }\delta>0, \\
\infty, &\text{if }-1<\delta\leq 0.
\end{array} 
\right.
\end{align}

Next, let $k\geq 2$. For $2\leq j\leq k$, $d_{j}=\ny(\phi_{j})$  and is distributed as $Y(A_{\phi_{j}},\Gamma_{\phi_{j}})$.
Hence
\begin{align}
\label{pam3}
\E_{\bar{\mu}}\bigg[\sum_{j=2}^{k}d_{j}\mid d_{\phi}=k\bigg] 
&=\E_{\bar{\mu}}\bigg[\sum_{j=2}^{k}\ny(\phi_{j})\mid d_{\phi}=k\bigg]\nonumber\\
&=\sum_{j=2}^{k}\E_{\bar{\mu}}[\Gamma_{\phi_{j}}]\, \E_{\bar{\mu}}\big[\kappa(A_{\phi_{j}})\mid d_{\phi}=k\big]\nonumber\\
&=(1+\delta)\E_{\bar{\mu}}\bigg[\sum_{j=2}^{k}\kappa(A_{\phi_{j}})\mid d_{\phi}=k\bigg]\nonumber\\
&=(1+\delta)\bigg(\E_{\bar{\mu}}\bigg[\sum_{j=2}^{k}A_{\phi_{j}}^{-\frac{1}{2+\delta}}\mid d_{\phi}=k\bigg]-k+1\bigg).
\end{align}
There exists $k-1$ non-ordered i.i.d.\ uniform $[A_{\phi},1]$ random variables $\tilde{U}_{1},\ldots ,\tilde{U}_{k-1}$ such that the random variable $\sum_{j=2}^{k}A_{\phi_{j}}^{-\frac{1}{2+\delta}}$ given the event $\{d_{\phi}=k\}$ is distributed as $\sum_{j=1}^{k-1}\tilde{U}_{j}^{-\frac{1}{2+\delta}} \mid d_{\phi}=k$. Therefore, using \eqref{pam3}, we can write
\begin{align}
\label{pam1}
\E_{\bar{\mu}}\bigg[\sum_{j=2}^{k}d_{j}\mid d_{\phi}=k\bigg] 
&=(1+\delta)\bigg(\E_{\bar{\mu}}\bigg[\sum_{j=1}^{k-1}\tilde{U}_{j}^{-\frac{1}{2+\delta}}\mid d_{\phi}=k\bigg]-k+1\bigg)\nonumber\\
&=(1+\delta)(k-1)\bigg(\E_{\bar{\mu}}\big[\tilde{U}_{1}^{-\frac{1}{2+\delta}}\mid d_{\phi}=k\big]-1\bigg).
\end{align}
Using \eqref{a-04} and \eqref{a-4}, we compute
\begin{align*}
&\E_{\bar{\mu}}\big[\tilde{U}_{1}^{-\frac{1}{2+\delta}}\mid d_{\phi}=k\big]\\
&= \int_{[0,1]}\dd u\, \int_{[0,u]}  \dd a\, \dfrac{\P_{\bar{\mu}}\{d_{\phi}=k\mid A_{\phi}=a\}}{\P_{\bar{\mu}}
\{d_{\phi}=k\}}(1-a)^{-1}u^{-\frac{1}{2+\delta}}\\
&= \dfrac{\gamma^\dagger_{k,\delta}}{2+\delta}
\int_{[0,1]}\dd a\, (1-a^{\frac{1}{2+\delta}})^{k-1}a^{\frac{1+\delta}{2+\delta}}(1-a)^{-1}\int_{[a,1]}  \dd u\, u^{-\frac{1}{2+\delta}}\\
&= \dfrac{\gamma^\dagger_{k,\delta}}{1+\delta}
\int_{[0,1]}\dd a\, (1-a^{\frac{1}{2+\delta}})^{k-1}a^{\frac{1+\delta}{2+\delta}}(1-a)^{-1}(1-a^{\frac{1+\delta}{2+\delta}})\\&
\leq \dfrac{\gamma^\dagger_{k,\delta}}{1+\delta}
\int_{[0,1]}\dd a\, (1-a^{\frac{1}{2+\delta}})^{k-1}a^{\frac{1+\delta}{2+\delta}}\\
&=\dfrac{2+\delta}{1+\delta}\,\gamma^\dagger_{k,\delta}
\int_{[0,1]}\dd u\, (1-u)^{k-1}u^{2+2\delta}
=\dfrac{2+\delta}{1+\delta}.
\end{align*}
Therefore, from \eqref{pam1} we get
\begin{align}
\label{pam2}
0\leq\E_{\bar{\mu}}\bigg[\sum_{j=2}^{k}d_{j}\mid d_{\phi}=k\bigg] \leq k-1 .
\end{align}
Observe that
\begin{align*}
\E_{\bar{\mu}}\big[\Delta_{\phi}\big]
&= \sum_{k=1}^{\infty}\P_{\bar{\mu}}\{d_{\phi}=k\}\bigg(\dfrac{1}{k}
\sum_{j=1}^{k}\E_{\bar{\mu}}\big[d_{j}\mid d_{\phi}=k\big]\bigg)-1.
\end{align*}
Hence from \eqref{a-7}, \eqref{a-6} and \eqref{pam2} we get 
with $\E_{\bar{\mu}}[\Delta_{\phi}]=\infty$ for $-1<\delta\leq 0$, and 
\begin{align*}
\dfrac{2+\delta}{\delta}\Big(\frac{1}{2}+\E_{\bar{\mu}}[d_{\phi}^{-1}]\Big)-\big(1-\E_{\bar{\mu}}[d_{\phi}^{-1}]\big)
\leq\E_{\bar{\mu}}\big[\Delta_{\phi}\big]\leq \dfrac{2+\delta}{\delta}\Big(\frac{1}{2}+\E_{\bar{\mu}}[d_{\phi}^{-1}]\Big), \qquad \delta>0,
\end{align*}
where, by \eqref{a-4},
\begin{align*}
\E_{\bar{\mu}}[d_{\phi}^{-1}] = \sum_{k \in \N} \dfrac{(2+\delta)\Gamma(3+2\delta)\Gamma(k+\delta)}{k\Gamma(1+\delta)\Gamma(k+3+2\delta)}.
\end{align*}
This completes the proof for the first moment. Now we derive the second moment. 
For $-1<\delta\leq 0$, we have $\E_{\bar{\mu}}[\Delta_{\phi}^{2}]=\infty$ because $\E_{\bar{\mu}}[\Delta_{\phi}]=\infty$. Let us therefore consider $\delta> 0$. First we observe that
\begin{align*}
\E_{\bar{\mu}}\big[d_{\phi}^{2}\big] = \E_{\bar{\mu}}\big[\big(1+\ny(\phi)\big)^{2}\big] 
&= 1+\E_{\bar{\mu}}[\Gamma_{\phi}^{2}]\E_{\bar{\mu}}[\kappa^{2}(A_{\phi})]+3
\E_{\bar{\mu}}[\Gamma_{\phi}]\E_{\bar{\mu}}[\kappa(A_{\phi})] = \tfrac{4}{\delta}+6.
\end{align*}
Also, since
\begin{align*}
\E_{\bar{\mu}}\Big[U_{\phi_{1}}^{-\frac{2}{1+\delta}}\Big]
=\left\{
\begin{array}{ll}
\frac{1+\delta}{-1+\delta} , &\text{if } \delta>1, \\
\infty, &\text{if } 0<\delta\leq 1 ,
\end{array} 
\right.
\end{align*}
for $k \in \N$ we have
\begin{align*}
\E_{\bar{\mu}}\big[\kappa^{2}(A_{\phi_{1}})\mid d_{\phi}=k\big]
&=\E_{\bar{\mu}}\big[U_{\phi_{1}}^{-\frac{2}{1+\delta}}\big]\
E_{\bar{\mu}}\big[A_{\phi}^{-\frac{2}{2+\delta}}\mid d_{\phi}=k\big]\\
&\quad -2\E_{\bar{\mu}}\big[U_{\phi_{1}}^{-\frac{1}{1+\delta}}\big]
\E_{\bar{\mu}}\big[A_{\phi}^{-\frac{1}{2+\delta}}\mid d_{\phi}=k\big]+1\\
&\left\{\begin{array}{ll}
<\infty , &\text{if } \delta>1, \\
=\infty, &\text{if } 0<\delta\leq 1,
\end{array} 
\right.
\end{align*} 
which implies that
\begin{align*}
\E_{\bar{\mu}}\big[d_{1}^{2}\mid d_{\phi}=k\big]
\left\{\begin{array}{ll}
<\infty , &\text{if } \delta>1, \\
=\infty, &\text{if } 0<\delta\leq 1.
\end{array}
\right.
\end{align*}
Moreover, if $k\geq 2$, then from \eqref{pam3} and \eqref{pam2} we get that, for every $2\leq j\leq k$,
\begin{align*}
\E_{\bar{\mu}}\big[d_{j}^{2}\mid d_{\phi}=k\big]
&\leq\E_{\bar{\mu}}\bigg[\sum_{j=2}^{k}d_{j}^{2}\mid d_{\phi}=k\bigg]\\
&= \sum_{j=2}^{k}\E_{\bar{\mu}}[\Gamma_{\phi_{j}}]\, \E_{\bar{\mu}}\big[\kappa(A_{\phi_{j}})
\mid d_{\phi}=k\big]+\E_{\bar{\mu}}[\Gamma_{\phi_{2}}^{2}]\,\E_{\bar{\mu}}
\bigg[\sum_{j=2}^{k}\kappa^{2}(A_{\phi_{j}})\mid d_{\phi}=k\bigg]\\
&\leq k-1+(k-1)\,\E_{\bar{\mu}}[\Gamma_{\phi_{2}}^{2}]\,\E_{\bar{\mu}}
\Big[\tilde{U}_{1}^{-\frac{2}{2+\delta}}-2\tilde{U}_{1}^{-\frac{1}{2+\delta}}+1\mid d_{\phi}=k\Big]
<\infty.
\end{align*}
Consequently, if $\delta>1$, then by the Cauchy-Schwarz inequality we have $\E_{\bar{\mu}}[d_{i}d_{j}\mid d_{\phi}=k]<\infty$ for every $k\geq 1$ and $1\leq i,j\leq k$. Writing $\nu_{2}(k)=\P_{\bar{\mu}}\{d_{\phi}=k\}$, we obtain
\begin{align*}
&\E_{\bar{\mu}}\big[\Delta_{\phi}^{2}\big] = \sum_{k\in\N}\nu_{2}(k) \E_{\bar{\mu}}\big[\Delta_{\phi}^{2}\mid d_{\phi}=k\big]
=\sum_{k\in\N}\nu_{2}(k) \E_{\bar{\mu}}\bigg[\Big(\frac{1}{k}\sum_{j=1}^{k}d_{j}+1-k\Big)^{2}~\Big |~ d_{\phi}=k\bigg]\\
&=\sum_{k\in\N}\nu_{2}(k)\E_{\bar{\mu}}\bigg[\frac{1}{k^{2}}\sum_{j=1}^{k}d_{j}^{2}+\frac{1}{k^{2}}
\sum_{i=1}^{k}\sum_{j=1 \atop {j\neq i}}^{k}d_{i}d_{j}+1+k^{2}+\dfrac{2}{k}\sum_{j=1}^{k}d_{j}
-2\sum_{j=1}^{k}d_{j}-2k~\Big |~ d_{\phi}=k\bigg]\\
&=\sum_{k\in\N}\sum_{j=1}^{k}\frac{1}{k^{2}}\nu_{2}(k)\E_{\bar{\mu}}\big[d_{j}^{2}\mid d_{\phi}=k\big]
+\sum_{k\in\N}\sum_{i=1}^{k}\sum_{j=1 \atop {j\neq i}}^{k}\frac{1}{k^{2}}\nu_{2}(k)\E_{\bar{\mu}}
\big[d_{i}d_{j}\mid d_{\phi}=k\big]+1+\E_{\bar{\mu}}\big[d_{\phi}^{2}\big]\\
&\quad +\sum_{k\in\N}\sum_{j=1}^{k}\frac{2}{k}\nu_{2}(k)
\E_{\bar{\mu}}\big[d_{j}\mid d_{\phi}=k\big]-2\sum_{k\in\N}\sum_{j=1}^{k}\nu_{2}(k)
\E_{\bar{\mu}}\big[d_{j}\mid d_{\phi}=k\big]-2\E_{\bar{\mu}}\big[d_{\phi}\big].
\end{align*}
Therefore $\E_{\bar{\mu}}[\Delta_{\phi}^{2}]$ is finite if and only if $\delta>1$.


\paragraph{Proof of Theorem~\ref{thm3PAM}.}

Let $\kappa(\cdot)$ and $Y(\cdot,\cdot)$ be as in the proof of Theorem \ref{thm1PAM}. Also, let $d_{i}$ denote the number of children of $\phi_{i} = \phi i$. For $j \in \N$ and $a_{1}\in[0,1]$, we have
\begin{align*}
\P_{\bar{\mu}}\big\{d_{1}=j\mid A_{\phi_{1}}= a_{1}\big\}
&= \int_{0}^{\infty} \dd b\,\P_{\bar{\mu}}\big\{\ny(\phi_{1})=j-1\mid A_{\phi_{1}}=a_{1},\Gamma_{\phi_{1}}=b\big\}
\dfrac{b^{1+\delta} \ee^{-b}}{\Gamma(2+\delta)}\nonumber\\
&=  \int_{0}^{\infty} \dd b\,\P\big\{Y(a_{1},b)=j-1\big\}\dfrac{b^{1+\delta} 
\ee^{-b}}{\Gamma(2+\delta)}\nonumber\\
&=\dfrac{\Gamma(j+1+\delta)}{\Gamma(j)\Gamma(2+\delta)}\Big(1-a_{1}^{\frac{1}{2+\delta}}\Big)^{j-1}a_{1}.
\end{align*}
Also, the joint density of $(A_{\phi_{1}},A_{\phi})$ is
\begin{align*}
f_{(A_{\phi},A_{\phi_{1}})}(a,a_{1})=\dfrac{1+\delta}{2+\delta}\,a^{-\frac{1+\delta}{2+\delta}}\,
a_{1}^{-\frac{1}{2+\delta}}, \qquad a\in[0,1],~a_{1}\in[0,a].
\end{align*}
Using \eqref{a-04} and \eqref{a-4}, we get that, for $j,k \in \N$,
\begin{align}
\label{afterr1}
&\P_{\bar{\mu}}\big\{d_{1}=j\mid d_{\phi}= k\big\}\nonumber\\
&=\int_{[0,1]}\dd a\,\dfrac{\P_{\bar{\mu}}\big\{d_{\phi}=k\mid A_{\phi}= a\big\}}{\P_{\bar{\mu}}
\big\{d_{\phi}=k\big\}}\int_{[0,a]}\dd a_{1}\,\P_{\bar{\mu}}\big\{d_{1}=j\mid A_{\phi_{1}}= a_{1}\big\}
f_{(A_{\phi},A_{\phi_{1}})}(a,a_{1})\nonumber\\
&=\dfrac{\Gamma(k+3+2\delta)\Gamma(j+1+\delta)}{(2+\delta)^{2}\Gamma(1+\delta)\Gamma(3+2\delta)\Gamma(k)\Gamma(j)}
\int_{[0,1]}\dd a\,\Big(1-a^{\frac{1}{2+\delta}}\Big)^{k-1}
\int_{[0,a]}\dd a_{1}\,\Big(1-a_{1}^{\frac{1}{2+\delta}}\Big)^{j-1}a_{1}^{\frac{1+\delta}{2+\delta}}\nonumber\\
&=\dfrac{\Gamma(k+3+2\delta)\Gamma(j+1+\delta)}{\Gamma(1+\delta)\Gamma(3+2\delta)\Gamma(k)\Gamma(j)}
\int_{[0,1]}\dd v\,\big(1-v\big)^{k-1}v^{1+\delta}\int_{[0,a^{\frac{1}{2+\delta}}]}\dd u\,\big(1-u\big)^{j-1}u^{2+2\delta}.
\end{align}


\paragraph{Lower bound.} 
We first show that, for $k,x$ large enough,
\begin{align}
\label{d1lowerdelneq0}
\P_{\bar{\mu}}\big\{d_{1}\geq x\mid d_{\phi}=k\big\}
\geq \dfrac{2^{-(1+\delta)}}{\Gamma(3+\delta)\Gamma(4+2\delta)}
\, x^{-(1+\delta)},
\end{align}
and
\begin{align}
\label{d1lowerdelneq}
&\P_{\bar{\mu}}\big\{d_{1}\geq k^{2}(x+1)\mid d_{\phi}=k\big\}
\ge \dfrac{2^{-(1+\delta)}}{\Gamma(3+\delta)\Gamma(4+2\delta)}\, k^{-(1+\delta)} x^{-(1+\delta)}. 
\end{align}
Indeed, noting that $\mathrm{e}^{-1}\leq\liminf_{j\to \infty}(1-u)^{j-1}\mathbbm{1}_{[0,\frac{1}{j-1}]}(u)\leq 1$, we have
\begin{align*}
&\quad ~\liminf_{k\to\infty}\big(1-v\big)^{k-1}\,\mathbbm{1}_{[0,1]}(v)\geq\limsup_{k\to\infty}\mathbbm{1}_{[0,\frac{1}{k-1}]}(v),\\
&\liminf_{j\to \infty}\big(1-u\big)^{j-1}\,\mathbbm{1}_{[0,a^{\frac{1}{2+\delta}}]}(u)\geq\limsup_{j\to \infty}\mathbbm{1}_{[0,\frac{1}{j-1}]}(u).
\end{align*}
Also, via Stirling's formula we have
\begin{align}
\label{stirling}
\lim_{y\to\infty} \frac{\Gamma(y+\alpha)}{y^{\alpha}\,\Gamma(y)} = 1, \qquad\alpha\in\R.
\end{align}
Hence, we obtain that, for $j,k$ large enough,
\begin{align*}
&\P_{\bar{\mu}}\big\{d_{1}=j\mid d_{\phi}= k\big\}\\
&\geq\dfrac{1}{\Gamma(1+\delta)\Gamma(3+2\delta)}\,k^{3+2\delta}\,j^{1+\delta}\,
\int_{[0,\frac{1}{k-1}]}\dd v\,v^{1+\delta}\int_{[0,\frac{1}{j-1}]}\dd u\,u^{2+2\delta}\\
&\geq\dfrac{1}{(2+\delta)\Gamma(1+\delta)\Gamma(4+2\delta)}
\, k^{1+\delta}\,j^{-(2+\delta)}.
\end{align*}
This, together with \eqref{truncatedRiemann} implies that, for $k,x$ large enough,
\begin{align*}
\P_{\bar{\mu}}\big\{d_{1}\geq x\mid d_{\phi}=k\big\}
&\geq \dfrac{1}{(2+\delta)\Gamma(1+\delta)\Gamma(4+2\delta)}
\, k^{1+\delta}\sum_{j=\lfloor 2x\rfloor}^{\infty} j^{-(2+\delta)}\\
&\geq \dfrac{2^{-(1+\delta)}}{\Gamma(3+\delta)\Gamma(4+2\delta)}
\, x^{-(1+\delta)},
\end{align*}
which completes the proof of \eqref{d1lowerdelneq0}, and
\begin{align*}
\P_{\bar{\mu}}\big\{d_{1}\geq k^{2}(x+1)\mid d_{\phi}=k\big\}
&\geq
\dfrac{1}{(2+\delta)\Gamma(1+\delta)\Gamma(4+2\delta)}
\, k^{1+\delta}\sum_{j=\lfloor 2k^{2}x\rfloor}^{\infty} j^{-(2+\delta)}\\&
\geq \dfrac{2^{-(1+\delta)}}{\Gamma(3+\delta)\Gamma(4+2\delta)}
\, k^{-(1+\delta)} x^{-(1+\delta)},
\end{align*}
which completes the proof of \eqref{d1lowerdelneq}.

Using \eqref{a-4} and \eqref{d1lowerdelneq} we get, for $x$ large enough,
\begin{align*}
\P_{\bar{\mu}}\{\Delta_{\phi}\geq x\} 
&=\sum_{k\in\N}\P_{\bar{\mu}}\big\{d_{\phi}=k \big\}
\P_{\bar{\mu}}\Big\{\sum_{j=1}^{k}d_{j}\geq k x +k(k-1)\mid d_{\phi}=k\Big\}\\
&\geq \sum_{k\in\N}\P_{\bar{\mu}}\big\{d_{\phi}=k \big\}
\P_{\bar{\mu}}\Big\{d_{1}\geq k^{2}(x+1)\mid d_{\phi}=k\Big\}\\
&\geq
C_{1,\delta}\,x^{-(1+\delta)},
\end{align*}
where
\begin{align*}
&C_{1,\delta} \leq \dfrac{(2+\delta)\Gamma(3+2\delta)\,2^{-(1+\delta)}}{\Gamma(1+\delta)
\Gamma(3+\delta)\Gamma(4+2\delta)}
\sum_{k\in\N}k^{-(1+\delta)} \dfrac{\Gamma(k+\delta)}{\Gamma(k+3+2\delta)}\in(0,\infty).
\end{align*}

\paragraph{Upper bound.} 

The proof of the upper bound requires more intricate bounds. 

First we show that, for $j,k\in\N$,
\begin{align}
\label{claim1}
 \P_{\bar{\mu}}\big\{d_{1}=j\mid d_{\phi}= k\big\} 
\le \dfrac{(1+\delta)\Gamma(j+1+\delta)\Gamma(k+3+2\delta)}{\Gamma(j+3+2\delta)\Gamma(k+2+\delta)}.
\end{align}
Indeed, for $j,k\in\N$, by \eqref{afterr1},
\begin{align*}
&\P_{\bar{\mu}}\big\{d_{1}=j\mid d_{\phi}= k\big\}\\
&\leq\dfrac{\Gamma(k+3+2\delta)\Gamma(j+1+\delta)}{\Gamma(1+\delta)\Gamma(3+2\delta)\Gamma(k)\Gamma(j)}
\int_{[0,1]}\dd v\,\big(1-v\big)^{k-1}v^{1+\delta}\int_{[0,1]}\dd u\,\big(1-u\big)^{j-1}u^{2+2\delta}\\
&=\dfrac{(1+\delta)\Gamma(j+1+\delta)\Gamma(k+3+2\delta)}{\Gamma(j+3+2\delta)\Gamma(k+2+\delta)},
\end{align*}
which proves \eqref{claim1}.

Next, we provide an upper bound for $\P_{\bar{\mu}}\big\{d_{1}\geq x\mid d_{\phi}=k\big\}$.
Abbreviate
\begin{equation*}
\tilde\gamma_{k,\delta} = \dfrac{(1+\delta)\Gamma(k+3+2\delta)}{\Gamma(k+2+\delta)}
\end{equation*}
Using \eqref{truncatedRiemann2}, \eqref{stirling} and \eqref{claim1}, we see that, for $k\in\N$ and $x$ large enough,
\begin{align}
\label{d1upperdelneq}
\P_{\bar{\mu}}\big\{d_{1}\geq x\mid d_{\phi}=k\big\}
&\leq \tilde\gamma_{k,\delta}
\sum_{j=\lceil x\rceil}^{\infty}\dfrac{\Gamma(j+1+\delta)}{\Gamma(j+3+2\delta)}\nonumber\\
&
\leq \tilde\gamma_{k,\delta}
\sum_{j=\lceil x\rceil}^{\infty}j^{-(2+\delta)} 
\leq \dfrac{\Gamma(k+3+2\delta)}{\Gamma(k+2+\delta)}(x-1)^{-(1+\delta)}.
\end{align}
From \eqref{a-4}, it can be easily checked that $\E_{\bar{\mu}}\big[\frac{\Gamma(d_{\phi}+3+2\delta)}{\Gamma(d_{\phi}+2+\delta)}\big]= C(\delta)\sum_{k\in\N}\frac{1}{(k+\delta)(k+1+\delta)}<\infty$ for some constant $C(\delta)$ depending on only $\delta$. Hence, \eqref{d1lowerdelneq0} and \eqref{d1upperdelneq} imply that 
\begin{align*}
\P_{\bar{\mu}}\big\{d_{1}\geq x\big\}\asymp x^{-(1+\delta)},\quad\quad x\rightarrow\infty .
\end{align*}

Next, we provide an upper bound for $\P_{\bar{\mu}}\Big\{\sum_{i=2}^{k}d_{i}\geq (k-1) x \mid d_{\phi}=k\Big\}$. To that end, fix $k\geq 2$, $2\leq i\leq k$ and $j\geq 0$.  For $a_{i}\in[0,1]$,
\begin{align*}
\P_{\bar{\mu}}\big\{d_{i}=j\mid A_{\phi_{i}}=a_{i}\big\}
&= \int_{0}^{\infty} \dd b_{i}\,\P_{\bar{\mu}}\big\{\ny(\phi_{i})=j\mid A_{\phi_{i}}=a_{i},\Gamma_{\phi_{i}}=b_{i}\big\}
\dfrac{b_{i}^{\delta} \ee^{-b_{i}}}{\Gamma(1+\delta)}\\
&= \int_{0}^{\infty} \dd b_{i}\,\P\big\{Y(a_{i},b_{i})=j\big\}
\dfrac{b_{i}^{\delta} \ee^{-b_{i}}}{\Gamma(1+\delta)}\\
&=\dfrac{\Gamma(j+1+\delta)}{\Gamma(j+1)\Gamma(1+\delta)}\Big(1-a_{i}^{\frac{1}{2+\delta}}\Big)^{j}
a_{i}^{\frac{1+\delta}{2+\delta}}.
\end{align*}
Also, for $a\in[0,1]$, conditioned on $d_{\phi}= k,A_{\phi}= a$, $A_{\phi_{i}}$ is distributed as the $(i-1)$-th order statistic of a sample of $k-1$ i.i.d.\ uniform$[a,1]$ random variables. The conditional density of $A_{\phi_{i}}$ given  $d_{\phi}= k$ and $A_{\phi}= a$ is given by
\begin{align*}
f_{A_{\phi_{i}}\mid d_{\phi}= k,A_{\phi}= a}(a_{i}) = \dfrac{\Gamma(k)}{\Gamma(i-1)\Gamma(k-i+1)}
a_{i}^{i-2}(1-a-a_{i})^{k-i}(1-a)^{-(k-1)}, \qquad a_{i}\in[a,1].
\end{align*}
Since
\begin{align*}
\dfrac{\Gamma(k)}{\Gamma(i-1)\Gamma(k-i+1)}=\frac{i(i-1)}{k}\binom{k}{i}\leq k\sum_{m=0}^{k}\binom{k}{m}= k\,2^{k},
\end{align*}
it follows that, for $a\in[0,1]$,
\begin{align*}
&\int_{[a,1]}\dd a_{i}\,\P_{\bar{\mu}}\big\{d_{i}=j\mid A_{\phi_{i}}=a_{i}\big\}f_{A_{\phi_{i}}\mid d_{\phi}= k,A_{\phi}= a}(a_{i})\\
&\leq \dfrac{ k\,2^{k}\,\Gamma(j+1+\delta)}{\Gamma(j+1)\Gamma(1+\delta)}(1-a)^{-i+1}
\int_{[a,1]}\dd a_{i}\,\Big(1-a_{i}^{\frac{1}{2+\delta}}\Big)^{j}a_{i}^{\frac{1+\delta}{2+\delta}}\\
&\leq \dfrac{(2+\delta)k\,2^{k}\,\Gamma(j+1+\delta)}{\Gamma(j+1)\Gamma(1+\delta)}(1-a)^{-i+1}
\int_{[0,1]}\dd u\,\big(1-u\big)^{j}u^{2+2\delta}\\
&\leq\dfrac{(2+\delta)k\,2^{k}\,\Gamma(j+1+\delta)\Gamma(3+2\delta)}{\Gamma(1+\delta)\Gamma(j+4+2\delta)}(1-a)^{-k+1}.
\end{align*}
Abbreviate
\begin{equation*}
\hat\gamma_{k,\delta} = \dfrac{2^{k}\,\Gamma(k+3+2\delta)}{\Gamma(1+\delta)\Gamma(k)}.
\end{equation*}
Using \eqref{a-04} and \eqref{a-4}, we obtain
\begin{align*}
&\P_{\bar{\mu}}\big\{d_{i}=j\mid d_{\phi}= k\big\}\\
&=\int_{[0,1]}\dd a\,\dfrac{\P_{\bar{\mu}}\big\{d_{\phi}=k\mid A_{\phi}= a\big\}}
{\P_{\bar{\mu}}\big\{d_{\phi}=k\big\}}\int_{[a,1]}\dd a_{i}\,
\P_{\bar{\mu}}\big\{d_{i}=j\mid A_{\phi_{i}}=a_{i}\big\}f_{A_{\phi_{i}}\mid d_{\phi}= k,A_{\phi}= a}(a_{i})\\
&\leq \dfrac{k\,\hat\gamma_{k,\delta}\,\Gamma(j+1+\delta)}{\Gamma(j+4+2\delta)}
\int_{[0,1]}\dd a\, a^{\frac{1+\delta}{2+\delta}}
\leq \dfrac{(2+\delta)k\,\hat\gamma_{k,\delta}\,\Gamma(j+1+\delta)}{(3+2\delta)\,\Gamma(j+4+2\delta)}.
\end{align*}
Hence, using \eqref{truncatedRiemann2} and \eqref{stirling}, for $k \geq 2$ and $x$ large enough, we obtain
\begin{align}
\label{diupperneq}
&\P_{\bar{\mu}}\Big\{\sum_{i=2}^{k}d_{i}\geq (k-1) x \mid d_{\phi}=k\Big\}
\leq \sum_{i=2}^{k}\P_{\bar{\mu}}\big\{d_{i}\geq x\mid d_{\phi}=k\big\}\nonumber\\
&\leq \sum_{i=2}^{k}\P_{\bar{\mu}}\big\{d_{i}\geq \lceil (x^{1+\delta}k^{2}\,2^{2k})^{\frac{1}{2+\delta}}\rceil +1\mid d_{\phi}=k\big\}\nonumber\\
&\leq k(k-1)\,\hat\gamma_{k,\delta}
\sum_{j=\lceil (x^{1+\delta}k^{2}\,2^{2k})^{\frac{1}{2+\delta}}\rceil +1}^{\infty}j^{-(3+\delta)}
\leq\dfrac{\Gamma(k+3+2\delta)}{\Gamma(1+\delta)\Gamma(k)}2^{-k}x^{-(1+\delta)}.
\end{align}

Similarly, abbreviating
\begin{align*}
&C_{2,\delta}=\dfrac{(2+\delta)\Gamma(3+2\delta)}{\Gamma(1+\delta)}
\bigg(\sum_{k\in\N}\dfrac{\Gamma(k+\delta)}{\Gamma(k+2+\delta)}+\dfrac{1}{\Gamma(1+\delta)}
\sum_{k\in\N\setminus\{1\}}\dfrac{\Gamma(k+\delta)}{\Gamma(k)}2^{-k}\bigg)\in(0,\infty),
\end{align*}
and using \eqref{a-4}, \eqref{d1upperdelneq} and \eqref{diupperneq}, by the dominated convergence theorem we get, for $x$ large enough,
\begin{align*}
&\P_{\bar{\mu}}\{\Delta_{\phi}\geq x\}\\
& \leq\sum_{k\in\N}\P_{\bar{\mu}}\big\{d_{\phi}=k \big\}
\P_{\bar{\mu}}\Big\{\sum_{j=1}^{k}d_{j}\geq k x \mid d_{\phi}=k\Big\}\\
&\leq \sum_{k\in\N}\P_{\bar{\mu}}\big\{d_{\phi}=k \big\}\P_{\bar{\mu}}\Big\{d_{1}\geq x \mid d_{\phi}=k\Big\}\\
&+\sum_{k\in\N\setminus\{1\}}\P_{\bar{\mu}}\big\{d_{\phi}=k \big\}\P_{\bar{\mu}}\Big\{\sum_{j=2}^{k}d_{j}\geq (k-1) x \mid d_{\phi}=k\Big\}
\leq C_{2,\delta}\,(x-1)^{-(1+\delta)}.
\end{align*}


\bibliographystyle{amsplain}
\bibliography{Reference.bib}


\end{document}